\documentclass[11pt]{amsart}
\usepackage{amsmath}
\usepackage{amsthm}
\usepackage{amssymb}
\usepackage[pdftex]{hyperref}
\usepackage[centering]{geometry}                
\usepackage{tikz}
\usetikzlibrary{decorations.markings}
\usetikzlibrary{arrows}
\usetikzlibrary{svg.path}
\usetikzlibrary{external}

\newtheorem{theorem}{Theorem}
\newtheorem{proposition}{Proposition}[section]
\newtheorem{corollary}[theorem]{Corollary}
\newtheorem{lemma}[proposition]{Lemma}

\theoremstyle{definition}
\newtheorem{remark}{Remark}

\title[Regularity conditions in the CLT]{Regularity conditions in the CLT for linear eigenvalue statistics of Wigner matrices}

\begin{document}

\begin{abstract}We show that the variance of centred linear statistics of eigenvalues of GUE matrices remains bounded for large $n$ for some classes of test functions less regular than Lipschitz functions. This observation is suggested by the limiting form of the variance (which has previously been computed explicitly), but it does not seem to appear in the literature. We combine this fact with comparison techniques following Tao-Vu and Erd\"os, Yau, et al. and a Littlewood-Paley type decomposition to extend the central limit theorem for linear eigenvalue statistics to functions in the H\"older class $C^{1/2+\epsilon}$ in the case of matrices of Gaussian convolution type. We also give a variance bound which implies the CLT for test functions in the Sobolev space $H^{1+\epsilon}$ and $C^{1-\epsilon}$ for general Wigner matrices satisfying moment conditions.  If the additional assumption of the test function being supported away from the edge of the spectrum is made, we prove the CLT for test functions of regularity $\dot{H}^{1/2}\cap L^\infty, H^{1/2+}$ and $H^{1-}$ for GUE, Johansson and Wigner matrices respectively. 
Previous results on the CLT impose the existence and continuity of at least one classical derivative.
\end{abstract}
\author{Philippe Sosoe and Percy Wong}
\address{Department of Mathematics, Princeton University, Fine Hall, Washington Road, Princeton NJ 08540, USA}
\email{psosoe@math.princeton.edu}
\email{pakwong@math.princeton.edu}
\date{\today}
\maketitle
\section{Introduction}
We are concerned with the fluctuations of linear eigenvalue statistics
\[\mathcal{N}_n[\varphi]=\sum_{j=1}^n\varphi(\lambda_j),\]
where $\lambda_j$, $1\le j \le n$ are the eigenvalues of a Hermitian Wigner matrix. 
The centred random variable
\begin{equation}\label{eq: centred}\mathcal{N}^\circ_n[\varphi]:=\mathcal{N}_n[\varphi]-\mathbf{E}(\mathcal{N}_n[\varphi])\end{equation}
exhibits strong cancellation, and it is known that for sufficiently smooth test functions, it converges to a normal random variable. This result is originally due to V. Girko, see \cite{girko}. In contrast, sums of $n$ centred i.i.d. random variables require a normalization of order $n^{-\frac{1}{2}}$ to obtain a non-trivial Gaussian limit. 

Some degree of smoothness is required of $\varphi$ for the asymptotic normality to hold without normalization, and the variance to remain bounded. Indeed, O. Costin and J. Lebowitz \cite{costinlebowitz} (see also \cite{gustafsson}) have shown that for GUE and for $y$ in the interior of the support of the limiting spectral density, and 
\[\varphi_y(x)=\mathbf{1}_{[y,\infty)}(x),\]
i.e. $\mathcal{N}_n[\varphi_y]$ is the number of eigenvalues greater than $y$, the variance has logarithmic asymptotic behaviour:
\[\mathbf{Var}(\mathcal{N}_n[\varphi])=\left(\frac{1}{2\pi^2}+o(1)\right)\log n.\]
Moreover, $\mathcal{N}_n^\circ[\varphi_y]/\sqrt{\mathbf{Var}(\mathcal{N}_n[\varphi_y)]}$ converges to a standard normal random variable. S. Dallaporta and V. Vu \cite{dallaportavu} have recently extended this result to Wigner matrices whose entries have sub-exponential tails. 

It is natural to ask what is the minimal degree of smoothness required of $\varphi$ for the variance to remain bounded, and for $\mathcal{N}^\circ_n[\varphi]$ to converge in distribution. The present work is motivated by this question.

In recent years, several authors have derived central limit theorems for centred linear eigenvalue statistics of Wigner matrices with various hypotheses, see \cite{andersonzeitouni}, \cite{cabanal}, \cite{chatterjee}, \cite{costinlebowitz}, \cite{dallaportavu}, \cite{guionnet}, \cite{johansson}, \cite{lytovapastur}, \cite{shcherbina}, \cite{sinaisoshnikov1}, \cite{sinaisoshnikov2}. In all these works, when general test functions are considered, these are required to be at least $C^1(\mathbb{R})$ (differentiable with a continuous derivative) and often much smoother. For example, A. Lytova and L. Pastur \cite{lytovapastur} have proved that for GOE matrices and $\varphi \in C^1(\mathbb{R})$ satisfying suitable decay conditions at infinity, $\mathcal{N}_n^\circ [\varphi]$ converges to a Gaussian random variable with mean zero and variance
\begin{equation}\label{eq: goevariance}
V_{\mathrm{GOE}}[\varphi] =\frac{1}{2\pi^2}\int_{-2}^2\int_{-2}^2\left(\frac{\varphi(x)-\varphi(y)}{x-y}\right)^2\frac{4-xy}{\sqrt{4-x^2}\sqrt{4-y^2}}\,\mathrm{d}x\mathrm{d}y.
\end{equation}
The variance $\mathbf{Var}(\mathcal{N}_n[\varphi])$ is controlled by a Poincar\'e-type bound specific to the Gaussian case.
Under a Lindeberg-type condition on the fourth moments, Lytova and Pastur extend their result (by comparison to the Gaussian case) to general real Wigner matrices assuming $\varphi \in C^5(\mathbb{R})$. For general symmetric Wigner matrices, there is an additional term in the expression for limiting variance on the right side of (\ref{eq: goevariance}), which vanishes if
\[\mathbf{E}|w_{ij}|^4=3\mathbf{E}|w_{ij}|^2,\]
where $w_{ij}=\bar{w}_{ij}$ are the entries of the Wigner matrix. One expects the fluctuations of linear statistics to be sensitive to the third and fourth moments of the entry distribution.

G. Anderson and O. Zeitouni \cite{andersonzeitouni} have developed a central limit theorem for linear statistics of $C^1(\mathbb{R})$ test functions for a large class of random matrices with independent, but not necessarily identically distributed entries which includes Wigner and Wishart matrices with entries satisfying a Poincar\'e inequality. 

Sobolev spaces offer an alternative to classical derivatives as a way to measure smoothness. For $s>0$, the (inhomogeneous) $L^2$-Sobolev space $H^s(\mathbb{R})$ is defined as the closure of the space of Schwartz functions $\mathcal{S}$ in the norm: \[\|f\|^2_{H^s}=\int_\mathbb{R}(1+|\xi|)^{2s}|\widehat{f}(\xi)|^2\,\mathrm{d}\xi,\] where $\widehat{f}$ is the Fourier transform of $f$.
K. Johansson \cite{johansson} considered Hermitian matrix models defined by their density with respect to Lebesgue measure on the entries:
\begin{gather*}\mathrm{d}\mu_n(\mathbf{x})=\frac{1}{Z}\exp(-\beta N\operatorname{Tr}V(\mathbf{x}))\,\mathrm{d}\mathbf{x},\\ \mathrm{d}\mathbf{x}=\prod_{i<j}\mathrm{d}\Re x_{ij}\mathrm{d}\Im x_{ij}\prod_{i}\mathrm{d}x_{ii}.\end{gather*}
It is well-known that for such matrices, the joint density on $\mathbb{R}^N$ of the eigenvalues has the form
\[\rho_{n,\beta}(\lambda_1,\ldots, \lambda_n) =\frac{1}{Z_{n,\beta}}\exp\left(-N\sum_{j=1}^N V(\lambda_j) + \frac{\beta}{2}\sum_{j\neq j}\log|\lambda_i-\lambda_j|\right),\]
and that the classical GUE ensemble corresponds to $\beta=2$, $V=2x^2$. Johansson obtains the central limit theorem for linear statistics for test functions of class $H^{2+\epsilon}(\mathbb{R})$ when $\beta=2$ and $H^{17/2+\epsilon}(\mathbb{R})$ when $\beta \neq 2$.\footnote{In \cite{johansson}, Remark 2.5, K. Johansson explains that the restrictions $\varphi \in H^{2+\epsilon}$ and $\varphi \in H^{17/2+\epsilon}$ are technical and that the ``correct condition should be the finiteness'' of the limiting expression for the variance.}
Based on an idea in \cite{johansson}, M. Shcherbina  \cite{shcherbina} (see also \cite{shcherbina2}) proved a variance bound which allowed her to extend the central limit theorem for linear statistics of real Wigner matrices to functions in the Sobolev space $H^{3/2+\epsilon}(\mathbb{R})$, for $\epsilon>0$ and  under weak assumptions on the moments of the entries. Note that such functions are in $C^1(\mathbb{R})$ by the Sobolev embedding. The key estimate in Shcherbina's work is the following bound for the variance of $\mathcal{N}_n[\varphi]$ (see \cite{shcherbina}, Prop. 1):
\begin{equation} \label{shcherbinasbound} \mathbf{Var}(\mathcal{N}_n[\varphi]) \le C_s\|\varphi\|^2_{H^s(\mathbb{R})}\int_0^\infty e^{-\eta}\eta^{2s-1}\int_{-\infty}^\infty\mathbf{Var}(\mathrm{Tr} G(E+i\eta))\,\mathrm{d}E\mathrm{d}\eta, \end{equation}
where 
\[G(E+i\eta)=G(z)=\frac{1}{H-z}\] 
is the resolvent matrix. It is the formula on the right side in (\ref{shcherbinasbound}) which motivated the Littlewood-Paley approach explained in Section \ref{sec: lp}. Equation (\ref{shcherbinasbound}) shows that it suffices to obtain a bound for the variance of  $\mathrm{Tr} G(z)$ of order $\eta^{-2s}$ to conclude that the variance of linear statistics can be bounded in terms of the $H^{s+\epsilon}$ norm of the test function for $\epsilon>0$. (The integration in $x$ turns out to be harmless).

In their work on universality for Wigner matrices, L. Erd\"os, H. T. Yau and their collaborators have obtained large deviation bounds of order $\eta^{-2}n^2\log n$ for the trace of the resolvent $\mathrm{Tr} G(z)$ for $\eta \gg n^{-1}$ (see the survey paper \cite{erdostucson}). We use their methods to show that
\begin{equation}\label{eq: variance}
\mathbf{Var}(\mathrm{Tr} G(z))\le C_\epsilon\eta^{-2-\epsilon}, \quad |\Re z|<5.
\end{equation}
Our result for linear statistics of functions in $H^{1+\epsilon}$ follows from this estimate:
\begin{theorem}\label{thm2}
Let $\varphi \in H^{1+\epsilon}(\mathbb{R})$, where $\epsilon>0$ is arbitrary. Define the random variable
\[\mathcal{N}_n[\varphi]=\sum_{j=1}^n\varphi(\lambda_j),\]
where $\lambda_j$, $1\le j\le n$ are the eigenvalues of a Hermitian Wigner matrix whose entries satisfy the following condition (``Condition \textbf{C0}''):
\begin{equation}
\label{eqn: Czero}
 \mathbf{P}(|w_{ij}|\ge t^{c})\le e^{-t}, \quad c>0.
 \end{equation}
There is a constant $C>0$ such that
\begin{equation}
\label{eq: h1variance}
\mathbf{Var}(\mathcal{N}_n[\varphi]) \le C\|\varphi\|^2_{H^{1+\epsilon}(\mathbb{R})}.
\end{equation}
In particular, $\mathcal{N}_n^\circ[\varphi]$ converges in distribution to a random variable with variance given by 
\begin{align*}&\frac{1}{4\pi^2}\int_{-2}^2\int_{-2}^2\left(\frac{\varphi(x)-\varphi(y)}{x-y}\right)^2\frac{4-xy}{\sqrt{4-x^2}\sqrt{4-y^2}}\,\mathrm{d}x\mathrm{d}y\\
+ &\frac{\kappa_4}{4\pi^2}\left(\int_{-2}^2\varphi(x)\frac{2-x^2}{\sqrt{4-x^2}}\,\mathrm{d}x\right)^2+\frac{w_2-2}{4\pi^2}\left(\int_{-2}^2\varphi(x)x\frac{2-x^2}{\sqrt{4-x^2}}\,\mathrm{d}x\right)^2,
\end{align*}
where 
\begin{align*}
\kappa_4&=\mathbf{E}|w_{ij}|^4-3\\
w_2 &= \mathbf{E}|w_{ij}|^2.
\end{align*}
\end{theorem}

Although the power of $\eta$ in \eqref{eq: variance} is essentially of the correct order, and Theorem \ref{thm1} improves significantly on previous results, the regularity assumption is not optimal. This is because the estimate \eqref{shcherbinasbound} is not sharp. In terms of our Littlewood-Paley expression for the variance \eqref{eq: equality}, this can be seen as a consequence of the application of Cauchy-Schwarz to the covariance in \eqref{eq: crucialCS}, which is wasteful. We expect the covariance between $\mathrm{Tr}G(z_1)$ and $\mathrm{Tr}G(z_2)$ to be much smaller than the square root of the product of variances as soon as $\Re z_1$ and $\Re z_2$ are separated. 

For GUE matrices, the limiting variance of the random variable $\mathcal{N}_n[\varphi]$ becomes (see \cite{cabanal}, or \cite{shcherbinapastur}, Ch. 5): 
\begin{equation}
\label{eq: guevariance}
V_{\mathrm{GUE}}[\varphi] =\frac{1}{4\pi^2}\int_{-2}^2\int_{-2}^2\left(\frac{\varphi(x)-\varphi(y)}{x-y}\right)^2\frac{4-xy}{\sqrt{4-x^2}\sqrt{4-y^2}}\,\mathrm{d}x\mathrm{d}y.
\end{equation} 
This integral is finite for $\alpha$-H\"older functions $\varphi$ of order $\alpha>1/2$ (see Section \ref{sec: chalfgue}). For $0<\alpha<1$, the space $C^\alpha(\mathbb{R})$ of  $\alpha$-H\"older functions consists of (continuous) functions such that the norm:
\begin{align*}
\|\varphi\|_{\alpha} = \|\varphi\|_{L^\infty(\mathbb{R})}+[\varphi]_{C^\alpha(\mathbb{R})}\\
[\varphi]_{C^{\alpha}}=\sup_{x\neq y}\frac{|\varphi(x)-\varphi(y)|}{|x-y|^\alpha}
\end{align*}
is finite. Since we are only concerned with functions defined on the real line, we use the notation $C^\alpha=C^\alpha(\mathbb{R})$ throughout. If we ignore the edges of the limiting spectrum, the integral in \eqref{eq: guevariance} becomes equivalent (as far as the regularity of $f$ is concerned) to
\[\int_{\mathbb{R}}\int_{\mathbb{R}}\frac{(f(x)-f(y))^2}{(x-y)^2}\,\mathrm{d}x\mathrm{d}y.\]
The last quantity is equal, up to a constant factor, to the homogeneous $\dot{H}^{1/2}(\mathbb{R})$ norm:
\[\|f\|_{\dot{H}^{\frac{1}{2}}}=\int_{\mathbb{R}}|\xi||\widehat{f}(\xi)|^2\,\mathrm{d}\xi.\]

One expects that also for finite $n$ the variance of $\mathbf{Var}(\mathcal{N}_n[\varphi])$ can be bounded in terms of the $C^{1/2+\epsilon}$ norm of $\varphi$, and by the $H^{1/2}(\mathbb{R})$ norm when $\varphi$ is supported away from the spectral edges $\pm 1$. This forms the basis of our next result:
\begin{theorem} \label{thm1} Let $\varphi\in \dot{H}^{1/2}\cap L^\infty(\mathbb{R})$. Define the random variable
\begin{equation}
\label{eqn: Nndefinition}
\mathcal{N}_{n}[\varphi] =\sum_{j=1}^n\varphi(\lambda_j),
\end{equation}
where $\lambda_j$ are the eigenvalues of a GUE matrix of size $n$.

If $\varphi$ is properly supported inside the bulk:
 \[\operatorname{supp} \varphi \subset (-2+\epsilon_0,2-\epsilon_0),\]
there is a constant $C(\epsilon_0)$ such that we have the estimate:
\[\mathbf{Var}(\mathcal{N}_n[\varphi])\le C\|\varphi\|^2_{\dot{H}^{1/2}\cap L^\infty(\mathbb{R})}.\] 

For functions $\varphi \in C^{1/2+\epsilon}$ (not necessarily compactly supported in the bulk), we have the following asymptotic result:
\begin{equation}
\label{eqn: guevariancebound}
\mathbf{Var}(\mathcal{N}_{n}[\varphi]) \rightarrow \frac{1}{4\pi^2}\int_{-2}^2\int_{-2}^2\left(\frac{\varphi(x)-\varphi(y)}{x-y}\right)^2\frac{4-xy}{\sqrt{4-x^2}\sqrt{4-y^2}}\,\mathrm{d}x\mathrm{d}y
\end{equation}
uniformly on bounded subsets of $C^{1/2+\epsilon}$. As a consequnce, there is a function a non-negative function $C(K)$, bounded on compact subsets of $\mathbb{R}$ such that
\[\mathbf{Var}(\mathcal{N}_{n}[\varphi])\le C(\|\varphi\|_{C^{1/2+\epsilon}})\|\|\varphi\|^2_{C^{1/2+\epsilon}}.\]

Moreover if $\varphi$ satisfies either of the two hypotheses above, the centred random variable
\[\mathcal{N}_{n}^\circ[\varphi]= \mathcal{N}_{n}[\varphi]-\mathbf{E}(\mathcal{N}_{n}[\varphi])\]
converges in distribution to a normal random variable of mean zero and variance given by (\ref{eq: guevariance}).
\end{theorem}

For functions supported away from the spectral edges $\pm 2$, the regularity threshold for asymptotic normality to hold without normalization is likely to be $\dot{H}^{1/2}$. This is consistent with alternate expressions for the variance found in \cite{cabanal} and \cite{johansson}, as well as the logarithmic divergence of the variance in case $\varphi =\mathbf{1}_I$, the indicator function of an interval $I$ in the interior of the limiting spectrum. Such functions barely fail to be in $\dot{H}^{1/2}$. Note in that in the case of Haar-distributed matrices on the unitary group, P. Diaconis and S. Evans \cite{diaconisevans} have obtained a sharp (with respect to the regularity of the test functions) central limit theorem. They find that the variance of the analogue of $\mathcal{N}_n[\varphi]$ grows like the $n$th partial sum of the appropriate $H^{1/2}$ norm. Their techniques rely on exact computations of matrix traces that are not available in the case of Wigner matrices. 

When the test function $\varphi$ is not supported away from the spectral edges, the expression \eqref{eq: guevariance} is no longer equivalent to the the $\dot{H}^{1/2}(\mathbb{R})$, which is why in this paper we use the $C^\alpha$ scale of spaces to measure regularity in case the support of the test function extends to the edge. We expect our result to be nearly optimal in this scale: i.e. the variance does not remain bounded for general test functions $\varphi \in C^{1/2-\epsilon}$ for any $\epsilon>0$. The limiting variance can be expressed in terms of an $\dot{H}^{1/2}$ norm defined relative to an orthogonal expansion of $\varphi$, restricted to the limiting specturm, into Chebyshev polynomials  (see \cite{johansson}). At the edges of the limiting spectrum, this expansion is not equivalent to the Fourier decomposition. This has to be taken into account to obtain a sharp result, and is the subject of current investigation.

By a moment matching argument and a saddle point analysis, we obtain a central limit theorem under similar assumptions than those made on $\varphi$ in the GUE case treated in \ref{thm1}, for more general Hermitian matrices, of the type considered by K. Johansson in \cite{johansson2}.
\begin{theorem}\label{thm3}
Let $W=(w_{ij})_{1\le i,j\le n}$ be a Hermitian Wigner matrix whose entries satisfy:
\[ \mathbf{P}(|w_{ij}|\ge t^{c})\le e^{-t}, \quad c>0,\]
and such that the first five moments of $w_{ij}$ match those of the GUE.  Form the Wigner matrix $M= \frac{1}{2\sqrt{n}}(W+V)$, and define the random variable $\mathcal{N}_n[\varphi]$ as in (\ref{eqn: Nndefinition}), but with $\lambda_j$ being the eigenvalues of $M$. If $\varphi\in C^{1/2+\epsilon}$ for $\epsilon>0$, or if $\varphi\in H^{1/2+\epsilon}$ and 
\[\operatorname{supp} \varphi \subset (-\sqrt{2}+\epsilon_0,\sqrt{2}-\epsilon_0)\] for some $\epsilon_0>$, then the random variable $\mathcal{N}_{n}^\circ[\varphi]$ converges in distribution to a mean zero normal random variable with variance given by 
\[\frac{1}{4\pi^2}\int_{-\sqrt{2}}^{\sqrt{2}}\int_{-\sqrt{2}}^{\sqrt{2}}\left(\frac{\varphi(x)-\varphi(y)}{x-y}\right)^2\frac{2-xy}{\sqrt{2-x^2}\sqrt{2-y^2}}\,\mathrm{d}x\mathrm{d}y.\]
\end{theorem}
The variance normalization in the previous theorem is of no particular significance. We have chosen it as a matter of convenience, for the calculations appearing in Sections \ref{sec: wigner} and \ref{sec: saddlept}. The strong assumption on the form of the matrix is needed to control the contribution to the variance from a certain range of frequencies of $\varphi$, but it is merely technical. It will be seen from the proof of Theorem \ref{thm3} that, although the variance of $\mathcal{N}_n[\varphi]$ admits an explicit expression in terms of (averages of) a determinantal kernal, we do not need very precise information on the asymptotics of this kernel to conclude. We believe that the control on the correlation between separated eigenvalues required to obtain our results for the GUE and Gaussian convolution matrices can be obtained for much more general Wigner matrices, without the use of explicit formulas.

In the final section, we derive the following central limit theorem for linear eigenvalue statistics for general Wigner matrices under matching moment conditions for regularity below one bounded derivative:
\begin{theorem}
Let 
\[H = \frac{1}{\sqrt{n}}W\]
be a Wigner matrix such that the entries of $W$ satisfy condition (\textbf{C0}) and such that the distribution of its entries matches that of GUE up to five moments.  Then for $\epsilon > 0$ sufficiently small, and any function $\varphi \in C^{1-\epsilon}$, the centred random variable
\[\mathcal{N}_{n}^\circ[\varphi]= \mathcal{N}_{n}[\varphi]-\mathbf{E}(\mathcal{N}_{n}[\varphi])\]
converges in distribution to a normal random variable of mean zero and variance given by (\ref{eq: guevariance}).  If the test function has support on $(-2+\epsilon_0,2-\epsilon_0)$ for some $\epsilon_0 > 0$, then the result holds for test functions $\varphi \in H^{1-\epsilon}$
\end{theorem}
The value of $\epsilon$ in the previous theorem can be estimated in terms of the number of matching moments of the entry distribution with the GUE. 

\subsection*{Acknowledgements} The first author would like to thank David Renfrew for pointing out M. Shcherbina's result \cite{shcherbina}, and for discussions that provided the initial motivation for the present work during the 2011 Prague Summer School on Mathematical Physics. We would also like to thank L\'aszl\'o Erd\"os for helpful comments on the variance estimates 
(\ref{shcherbinasbound}) and (\ref{eq: lpvariancebound}).

\section{Approximation by Poisson Integrals} \label{sec: approx}
To derive the central limit theorem for functions with limited smoothness, we bound the variance of $\mathcal{N}^\circ_n[\varphi]$ by the norm of $\varphi$ in a function space describing the regularity of $\varphi$. The central limit theorem then follows by a simple approximation argument. In what follows, we use the next lemma, adapted from \cite{shcherbinapastur}, Proposition 3.2.5 (see also \cite{lytovapastur}, p. 1793).
\begin{lemma}
Suppose $X$ is a normed vector subspace of $C_b(\mathbb{R})$, i.e. continuous functions bounded on $\mathbb{R}$, such that for $\|\varphi\|_X \le K$, we have the bound
\[V_n(\varphi):=\mathbf{Var}\left(\mathcal{N}_n[\varphi]\right) \le C(K) \|\varphi\|_X^2,\]
where $C:\mathbb{R}^+\rightarrow \mathbb{R}^+$ is bounded on compact subsets.
If there exists a dense subspace $Y\subset X$ such that for $\psi \in Y$, we have
\[Z_n[\varphi]:= \mathbf{E}(e^{ix\mathcal{N}^\circ_n[\varphi]}) =\exp(-x^2V(\psi)/2)+o(1),\]
for a continuous quadratic functional $V$ on $X$,
then the central limit theorem is valid for all $\varphi \in X$.
\begin{proof}
Let $\{\psi_k\}\subset Y$ converge to $\varphi \in X$. In particular, $\|\psi_k\|_X$ is a bounded sequence. Then:
\begin{align*}
|Z_n[\psi_k](x)-Z_n[\varphi](x)| &\le |x|\mathbf{E}\left[ |\mathcal{N}^\circ_n[\psi_k]-\mathcal{N}_n^\circ[\varphi]|\right]\\
&\le |x|\mathbf{Var}(\mathcal{N}_n[\varphi-\psi_k])^{1/2}\\
&\lesssim |x|\|\varphi-\psi_k\|_X.
\end{align*}
The result follows from the continuity of $V$ and the L\'evy continuity theorem.
\end{proof}
\end{lemma}

In her proof of the CLT for linear statistics for $H^{3/2+\epsilon}$ functions \cite{shcherbina}, M. Shcherbina used approximation by functions of the form: 
\[P_\eta \ast f(x)=\frac{1}{\pi}\int_{\mathbb{R}} \frac{\eta}{(x-t)^2+\eta^2}\, f(t)\,\mathrm{d}t,\] where $P_\eta$ is the Poisson kernel
\[P_y(x)=\frac{1}{\pi}\frac{\eta}{\eta^2+x^2}, \quad \eta >0.\]
When $f\in L^p(\mathbb{R})$ for $1\le p \le \infty$, $P_\eta\ast f$ is real analytic, so the validity of the CLT for such functions follows from work anterior to \cite{shcherbina}, see \cite{sinaisoshnikov1}. However, since $P_\eta(x) =\Im \frac{1}{x+i\eta}$, the CLT for Poisson integrals has a direct proof based on a calculation involving resolvents (see the proof of Theorem 1 in \cite{shcherbina}). 

We use the following standard approximation results for Poisson integrals:
\begin{lemma} \label{lem: hsapprox}
Let $f\in H^s(\mathbb{R})$, for $0\le s <\infty$. Then the Poisson integral $P_\eta\ast f$ converges in the $H^s$ norm to $f$ for $\eta\rightarrow 0$.
\end{lemma} 
\begin{lemma} \label{lem: lipapprox}
Let $f \in C^{\alpha}(\mathbb{R})$ for some $0<\alpha <1$, and $\beta>0$ be less than $\alpha$. The Poisson integral $P_\eta \ast f$ converges to $f$ in the $C^\beta$ norm for $\eta\rightarrow 0$.
\end{lemma}
The first result is a simple consequence of the explicit form of the Fourier transform of $P_\eta \ast f$:
\[\widehat{P_\eta\ast f}(\xi)=e^{-\eta|\xi|}\widehat{f}(\xi).\]
From this, we have:
\[\|P_\eta\ast f-f\|^2_{H^s}=\int(1+|\xi|)^{2s}(1-e^{-\eta|\xi|})^2|\widehat{f}(\xi)|^2\,\mathrm{d}\xi,\]
and the Lemma \ref{lem: hsapprox} follows from the dominated convergence theorem.

We indicate how to derive Lemma \ref{lem: lipapprox}. We begin with the $L^\infty$ part of the norm:
\begin{align*}
|(P_\eta\ast f)(x)-f(x)|&\le \int_{\mathbb{R}}|P_\eta(t)|\cdot |f(x-t)-f(x)|\,\mathrm{d}t \\
&\le \int_{\mathbb{R}} |P_\eta(t)||t|^\alpha\,\mathrm{d}t.\\
&\lesssim \eta^\alpha,
\end{align*}
so $P_\eta\ast f$ converges uniformly to $f$ as $\eta$ tends to 0.
The convergence of the seminorm $[P_y\ast f]_\alpha$ can be proved using Littlewood-Paley techniques, or using the characterization found in \cite[Section V.4]{stein}; for $f\in C^\alpha$, the seminorm $[f]_\alpha$ is equivalent to the least $M$ such that
\[\|\partial_\eta P_\eta\ast f\|_\infty \le M\eta^{-1+\alpha}.\]
Thus 
\begin{align*}
[P_\eta\ast f - f]_\beta &\lesssim \eta^{1-\beta}\|\partial_\eta P_\eta\ast f\|_{\infty}\\
&\lesssim \eta^{\alpha-\beta}[f]_\alpha.
\end{align*}
The convergence follows from this.

\section{Littlewood-Paley Theory}
\label{sec: lp}
In this section, we derive a representation for the variance of $\mathcal{N}_n[\varphi]$ based on a decomposition of $\varphi$ of Littlewood-Paley type. 
\begin{theorem}\label{thm: representation} Let $0\le s<\infty$ and $0<\beta<\alpha<1$. Then if $\varphi \in H^s$ or $\varphi \in C^\alpha$, $\varphi$ has a representation:
\[\varphi(x) = \sum_{k=-1}^\infty (P_{2^{-k}}\ast g_k)(x),\]
where the sum on the right is convergent in $H^s$ if $\varphi \in H^s$, and in $C^\beta$ if $\varphi \in C^\alpha$. The functions $g_k$, $-1\le k <\infty$ are smooth and such that, for some constants $C_s$ and $C_{\alpha,\beta}$:
\[\sum_{k=-1}^\infty 2^{2ks}\|g_k\|^2_{L^2} \le C_s\|\varphi\|^2_{H^s} \]
and
\[ \sum_{k=-1}^\infty 2^{\beta k}\|g_k\|_{L^\infty} \le C_{\alpha-\beta}\|\varphi\|_{C^\alpha}.\]
\end{theorem}
\begin{corollary}
Let $\varphi \in \mathcal{S}(\mathbb{R})$, the class of Schwartz functions, i.e. $C^\infty$ functions on $\mathbb{R}$ such that $\sup_{x\in \mathbb{R}} |x^mf^{(k)}(x)|\le C_{mk}<\infty$ for $m,k<\infty$. The variance of $\mathcal{N}_n[\varphi]$ can be written as
\begin{equation}\label{eq: equality}
\mathbf{Var}(\mathcal{N}_n[\varphi]) = \sum_{k,l=-1}^\infty\int_{\mathbb{R}}\int_{\mathbb{R}}g_k(t)\overline{g_l(s)}\mathbf{Cov}(\mathcal{N}_n[P_{2^{-k}}(\cdot - t)],\mathcal{N}_n[P_{2^{-l}}(\cdot -s)])\,\mathrm{d}t\mathrm{d}s.
\end{equation}
The sum on the right side of (\ref{eq: equality}) is absolutely convergent for $\varphi \in \mathcal{S}$.
\end{corollary}

As announced above, the main tool in the proof of Theorem \ref{thm: representation} will be Littlewood-Paley theory. A good general reference on this is \cite[Ch. 2]{bahourietal}. The dyadic frequency decomposition was introduced in the context of Fourier series by J.E. Littlewood and R.E.A.C. Paley \cite{littlewoodp1}, \cite{littlewoodp2}, \cite{littlewoodp3}. E. M. Stein developed extensions of the theory to functions on higher dimensional spaces $\mathbb{R}^d$ using various \emph{Littlewood-Paley functions} based on the Poisson kernel (see \cite[Chapter IV]{stein}). Here we will use the smooth dyadic frequency decomposition presented in \cite{bahourietal} (see also \cite{chemin}), and the description of the functions spaces $C^\alpha$ and $H^s$ in terms of this decomposition. We denote by
\[\widehat f(\xi) =\int e^{-i\xi x}f(x)\,\mathrm{d}x\]
the Fourier transform of $f\in \mathcal{S}'$, the space of tempered distributions.

Given a function $\varphi \in \mathcal{S}$, its \emph{inhomogeneous Littlewood-Paley decomposition} is the sum 
\[\varphi = \sum_{k\ge -1} \varphi_k,\]
with
\begin{align*}
\varphi_{-1}&=h\ast\varphi\\
\varphi_{k}&=2^k\omega(2^k\cdot)\ast\varphi, \quad k\ge 0,
\end{align*}
where $0\le \widehat h \le 1$ and $0\le \widehat \omega \le 1$ are even functions with compact support in the ball $\{|\xi|\le 3/4\}$ and the ``annulus'' $\mathcal{C}= \{3/4 \le |\xi| \le 8/3\}$, respectively, such that:
\begin{align*}
\widehat{h}(\xi)+ \sum_{k\ge 0}\widehat{\omega}(2^{-k}\xi)&=1\\
0\le \widehat{h}(\xi)^2+\sum_{k\ge 0}\widehat{\omega}^2(2^{-k}\xi)&\le 1.
\end{align*}
for $\xi\in \mathbb{R}$.
Note that $\varphi_k$ is frequency-localized in an annulus of size $2^k$, i.e.
\[\operatorname{supp} \,\widehat{\varphi_k} \subset 2^{k}\mathcal{C}.\]
The first component $\varphi_{-1}$ has its Fourier support contained in a ball \[\operatorname{supp} \,\widehat{\varphi_{-1}} \subset \{ |\xi|\le 3/4\}.\] For the existence of this decomposition and the corresponding dyadic partition of frequency space, see \cite{bahourietal}, Proposition 2.10.

The following theorem characterizes $H^s(\mathbb{R})$ and $C^\alpha(\mathbb{R})$ in terms of the Littlewood-Paley decomposition (see \cite{bahourietal}, Section 2.7):
\begin{theorem}\label{thm: lpcharacterization}
The inhomogeneous Sobolev space $H^s(\mathbb{R})$ consists of those tempered distributions such that
\[\|\varphi\|^2_{B^s_{2,2}}:= \sum_k 2^{2ks}\|\varphi_k\|_{L^2}^2 <\infty.\]
Moreover, the norms $\|\cdot\|_{B^s_{2,2}}$ and $\|\cdot\|_{H^s}$ are equivalent.
The space  $C^\alpha(\mathbb{R})$ of $\alpha$-H\"older continuous functions is the space of distributions such
\[\|\varphi\|_{B^\alpha_{\infty,\infty}}:= \sup_{k} 2^{\alpha k}\|\varphi_k\|_{L^\infty}<\infty.\]
The norms $\|\cdot\|_{B^\alpha_{\infty,\infty}}$ and $\|\cdot\|_{C^\alpha}$ are equivalent.
\end{theorem}

We turn to the proof of Theorem \ref{thm: representation} and its corollary. Given the support properties of $\widehat{\varphi_k}$, we can define $g_k$ by
\[\widehat{g_k} = e^{2^{-k}|\xi|}\widehat{\varphi_k}(\xi).\]
We have:
\[\|g_k\|_{L^p} \lesssim \|\varphi_k\|_{L^p}, \quad  p=2,\infty.\]
Indeed, for $k\ge 0$, let $\psi(2^{-k}\,\cdot)$ be a function in $C^\infty$ which is $1$ on the support of the Littlewood-Paley cutoff $\widehat{\omega}(2^{-k}\,\cdot)$. Then:
\[ g_k  =  \mathcal{F}^{-1}(e^{2^{-k}|\xi|} \psi(2^{-k}\xi))\ast\varphi_k.\]
$\mathcal{F}^{-1}$ denotes the inverse Fourier transform.
By Young's inequality (\cite{bahourietal}, Lemma 1.4)
\[\|g_k\|_{L^p}\le \| \mathcal{F}^{-1}(e^{2^{-k}|\xi|} \psi(2^{-k}\xi))\|_{L^1}\|\varphi_k\|_{L^p},\]
it suffices to obtain an $L^1$ bound for
\[ \mathcal{F}^{-1}(e^{2^{-k}|\xi|} \psi(2^{-k}\xi)) =\int e^{ix\xi}e^{2^{-k}|\xi|} \psi(2^{-k}\xi)\,\mathrm{d} \xi.\]
Let us compute:
\[\int \Big| \int e^{ix\xi}e^{2^{-k}|\xi|} \psi(2^{-k}\xi)\,\mathrm{d} \xi \Big|\mathrm{d}x.\]
Changing variables $s=2^{-k}\xi$, we have:
\[ \int \Big| \int e^{ix2^ks}e^{|s|} \psi(s)\,2^k\mathrm{d} s \Big|\mathrm{d}x. \]
Another change of variables (in the outer integral) $t=2^k x$ gives:
\[ \int \Big| \int e^{it s}e^{|s|} \psi(s)\,\mathrm{d} s \Big|\mathrm{d}t. \]
The integrand is $O(t^{-\infty})$ by repeated integration by parts. For the inhomogenous term, we have:
\[\|g_{-1}\|_{L^2} \lesssim \|\varphi_{-1}\|_{L^2}\]
by definition of the multiplier, and since the Fourier support is contained in $\{|\xi|\le 1/2\}$, we also have:
\[\|g_{-1}\|_{L^p} \lesssim  \|\varphi_{-1}\|_{L^2} \]
for any $p\ge 2$ by Bernstein's inequality \cite[Lemma 2.1]{bahourietal}.

Moreover, each of the $\varphi_k$ is in $\mathcal{S}$, and thus we have the convolution representation:
\[\varphi_k(x) =\frac{1}{\pi}\int g_k(t)\frac{2^{-k}}{(x-t)^2+(2^{-k})^2}\,\mathrm{d}t.\]
Thus we can write
\[\varphi = \sum_{k\ge -1} P_{2^{-k}}\ast g_k,\]
where the sum is convergent in $H^s$ if $\varphi\in H^s$. If $\varphi\in C^\alpha$, then we have
\begin{align*}
\|\varphi\|_{C^\beta}&\le \sum_k \|\varphi_k\|_{C^\beta} \\ 
&\lesssim \sum_k 2^{\beta k}\|\varphi_k\|_{L^\infty}\\
&\le \|\varphi\|_{C^\alpha} \sum_k 2^{(\beta-\alpha)k}.
\end{align*}
We insert this expression in the variance:
\begin{align*}
\mathbf{Var}(\mathcal{N}_n[\varphi]) &= \mathbf{E}((\sum_{k\ge -1} P_{2^{-k}}\ast g_k-\mathbf{E}\sum_{k\ge -1} P_{2^{-k}}\ast g_k)\overline{(\sum_{l\ge -1} P_{2^{-l}}\ast g_l-\mathbf{E}\sum_{l\ge -1} P_{2^{-l}}\ast g_l)})\\
&= \sum_{k,l\ge-1}\mathbf{Cov}(\mathcal{N}_n[P_{2^{-k}}\ast g_k]\, \overline{\mathcal{N}_n[P_{2^{-l}}\ast g_l]})\\
&=  \sum_{k,l\ge-1}\int \int g_k(t)\overline{g_l}(s) \mathbf{Cov}(P_{2^{-k}}(\cdot-t)\, \overline{P_{2^{-l}}(\cdot - s)})\,\mathrm{d}t\mathrm{d}s,
\end{align*}
which is formula (\ref{eq: equality}).
Thus, we may bound the variance as:
\begin{equation}
\label{eq: lpvariancebound}
\mathbf{Var}(\mathcal{N}_n[\varphi])\le \sum_{k,l} \int\int |g_k||g_l| \mathbf{Cov}(\mathcal{N}_n[P_{2^{-k}}(t-\cdot)],\mathcal{N}_n[P_{2^{-l}}(s-\cdot)])\mathrm{d}t\,\mathrm{d}s.
\end{equation}

The following estimate will be useful below to control contributions from very high frequencies. Its proof is a straightforward application of Theorem \ref{thm: lpcharacterization}:
\begin{lemma}\label{lem: hfcutoff}
Let $M\ge1$. Define the low-frequency cut-off operator $S_M\varphi$ by
\[S_M\varphi =\sum_{-1 \le k \le M-1} \varphi_k.\]
Then there is a constant $C_\alpha$ such that, for $\varphi\in C^\alpha$,
\[\|(1-S_M)\varphi\|_{L^\infty} \le C_\alpha 2^{-M\alpha}\|\varphi\|_{C^\alpha}. \]
\end{lemma}

\section{$H^{1+}$ CLT for general Wigner matrices}
The goal of this section is to prove the following bound:
\begin{proposition} \label{prop:varh1_res}
Let
\[H := \frac{1}{\sqrt{n}}W\]
be a Wigner matrix such that the entries of $H$ satisfy condition (\textbf{C0}). Let
\[z = E+i\eta\in\mathbf{T}_L,\]
where
\[\mathbf{T}_L := \{z = E + i\eta | |E| \leq 5, 0 < \eta \leq 10\}.\]
We have
\begin{equation} \label{eqn:varbound}
\emph{\textbf{Var}(\textrm{Tr}}G(z)) \lesssim \frac{1}{\eta^{2+\epsilon}},
\end{equation}
where $G =  (M-z)^{-1}$ is the resolvent for $M$.
\end{proposition}

To prove the proposition, we will make heavy use of the following theorem from \cite{EYY} regarding uniform rigidity of eigenvalues. Before stating it, we introduce two definitions: first, we let $m_{sc}(z)$ be the Stieltjes transform of the semicircle distribution on $[-2,2]$:
\[m_{sc}(z) =\frac{2}{\pi}\int_{-2}^2 \frac{1}{z-x}\sqrt{4-x^2}\, \mathrm{d}x.\]
We denote by $m(z)$ the normalized trace of the resolvent:
\[m(z) =\frac{1}{n}\operatorname{tr}(H-z)^{-1}.\]
\begin{theorem} [L. Erd\"os, E. Yau, J. Yin] \label{thm:EYY}
Let $H = (h_{ij})$ be a hermitian or symmetric $n\times n$ random matrix, $n \geq 3$, with $\mathbf{E}h_{ij} = 0$, $1 \leq i,j \leq n$, and assume that variances are given by $\sigma_{ij}^2 =\mathbf{E}|h_{ij}|^2 = \frac{1}{n}$.  Suppose that the distributions of the matrix elements satisfy condition \textbf{C0}.  There exist constants $A_0 > 1, \phi < 1$ such that for all $L > A_0 \log \log n$ and sufficiently large $n$,
\begin{equation}\label{eq: rigidity}
\mathbb{P}\left(\cup_{z\in\mathbf{S}_L}\{|m(z)-m_{sc}(z)| \geq \frac{(\log n)^{4L}}{n\eta} \}\right) \lesssim  \exp \left(-c(\log n)^{\phi L}\right),\end{equation}
where 
\[
\mathbf{S}_L := \{z = E + i\eta : ~|E| \leq 5, n^{-1}(\log n)^{10L} < \eta \leq 10\}.\]
For individual elements on the diagonal, we have, letting $\Lambda_d = \max_k |G_{kk}-m_{sc}|$, 
\[
\mathbb{P}(\cup_{z\in\mathbf{S}_L}\{\Lambda_d \geq \frac{(\log n)^{4L}}{\sqrt{n\eta}} \}) \lesssim \exp \left(-c(\log n)^{\phi L}\right).\]
\end{theorem}
\begin{proof}[Proof of Proposition \ref{prop:varh1_res}]
Directly from Theorem \ref{thm:EYY}, we can already conclude (\ref{eqn:varbound}) for the region 
\[\{z =   E + i\eta : ~|E| \leq 5,\   n^{-1+\delta} < \eta < n^{-1/10}\},\]
for any fixed $0<\delta<1$.  
We will now consider the regions $0<\eta < n^{-1+\delta}$ and $n^{-1/10} < \eta < 10$ separately, with
\[0<\delta < \frac{1}{2} \frac{\epsilon}{2+\epsilon}.\]

For the region $\eta < n^{-1+\delta}$, we will use the fact that the eigenvalues are close to their classical locations. Let $E = \Re z$, and write
\begin{equation} \label{eqn:resolvent}
\textrm{Tr}\,G(z) = \sum_{i=1}^n \frac{1}{\lambda_i-E-i\eta}.
\end{equation}
Consider a dyadic decomposition around $E$; define
\begin{equation*}
U_0 = \{j : |\lambda_j - E| \leq n^{-1+\alpha}\},
\end{equation*} and for $p = 1,\ldots, \log(n^{1-\alpha})$
\begin{equation*}
U_p = \{j : 2^{p-1}n^{-1+\alpha} \leq |\lambda_j - E| \leq 2^pn^{-1+\alpha}\} 
\end{equation*} and
\begin{equation*}
U_\infty = \{j : |\lambda_j - E| > 1\}.
\end{equation*}
$\alpha>0$ will be chosen below. Given a sequence $A_n$ of events, we say that $A_n$ holds with overwhelming probability if
\[\mathbf{P}(A_n^c)\le C_kn^{-k}\]
for any $k$. By the local semicircle law, we know that for any $\delta'=\epsilon/100$, we have $|U_p| \leq 2^pn^{\alpha+\delta'}$ with overwhelming probability. In particular
\[(\mathbf{E}|U_p|^2)^{1/2}\lesssim 2^pn^{\alpha+\delta'},\]
since $|U_p|\le n$. 
For eigenvalues with index in $U_p$, we have the bound 
\[\frac{1}{|\lambda_j-E|} \lesssim 2^{1-p}n^{1-\alpha}.\] In $U_0$, we use the trivial bound 
\[\frac{1}{|\lambda_j-E-i\eta|} \le \eta^{-1}.\] 
Split the sum (\ref{eqn:resolvent}) according to the partition $\{U_p\}$ to find
\begin{align*}
\mathbf{E}|\mathrm{Tr}G(z)|^2 &\lesssim \sum_{k,l\neq 0} n^{2-2\alpha}2^{-k}2^{-l}\mathbf{E}|U_k||U_l|\\
&\quad+\frac{n^{1-\alpha}}{\eta}\sum_k2^{-k}\mathbf{E}|U_0||U_k| +\frac{1}{\eta^2}\mathbf{E}|U_0|^2.
\end{align*}
Applying Cauchy-Schwarz in every term, we can bound this, up to a factor, by $(n)^{2+2\delta'}(\log n)^2 + n^{2\alpha+2\delta'}/\eta^2$.  Since we are assuming $\eta < n^{-1+\delta}$, $\mathbf{E}|\textrm{Tr}\,G(z)|^2$ is thus bounded up to a constant factor by $\eta^{-2-\epsilon}$, provided we choose $\delta'<\epsilon/4$, and $\alpha$ such that
\[(\alpha+\epsilon/4)\cdot\left(1+\frac{\delta}{1-\delta}\right)<\frac{\epsilon}{2}.\]

We turn to the case where $n^{-1/10} < \eta < 10$.  In this case, we cannot obtain (\ref{eqn:varbound}) from a direct application of Theorem \ref{thm:EYY}: we cannot avoid the $\log$ factor appearing in large deviation-type bounds. To circumvent the problem, we estimate the variance of the resolvent trace using the set of self-consistent equations. 

We first introduce some definitions found in \cite{EYY}. Let 
\[G_{ij}:=(H-z)^{-1}(i,j), \quad 1\le i,j\le n.\]
Define $H^{(i)}$ to be the $(n-1)\times (n-1)$ Wigner matrix obtained from $H$ by removing the $i$th row and column. We set
\[G^{(i)}_{ij}=(H^{(i)}-z)^{-1}(i,j).\]
The vector $\mathbf{a}^i= (\mathbf{a}^i_k)_{1\le k \le n}$ is obtained from the $i$th column of $H^{(i)}$ by deleting its $i$th element. We denote by $\mathbf{E}_i$ the expectation with respect to the $i$th column of $H$, i.e. with respect to the random variables $(h_{ji})_{1\le j\le n}$. 

\begin{lemma}
The diagonal resolvent matrix elements $G_{ii}$ satisfy the following system of self-consistent equations
\[G_{ii} = \frac{1}{-z-\sum_j\sigma_{ij}G_{jj} + \Upsilon_i},\]
where
\[\Upsilon_i := A_{i} + h_{ii} - Z_i\]
and
\begin{equation}\label{eqn:A}
A_i = \sigma_{ii}^2G_{ii}+\sum_{j\neq i}\sigma_{ij}^2\frac{G_{ij}G_{ji}}{G_{ii}}\end{equation}
\[Z_i = Z^{(i)}_{ii}-\mathbf{E}_iZ^{(i)}_{ii},\qquad Z_{ii}^{(i)}= \mathbf{a}^i\cdot G^{(i)}\mathbf{a}^i = \sum_{k,l\neq i} \mathbf{a}_k^i\cdot G^{(i)}_{kl}\bar{\mathbf{a}_l^{i}}.\]
\end{lemma}
In view of the previous lemma, we can write 
\[v_i:= G_{ii}-m_{sc} = \frac{1}{-z-m_{sc}-(\sum_j\sigma_{ij}^2v_j-\Upsilon_i)}-m_{sc}.\]  Using  the equation
\[(m_{sc}+z) = -m_{sc}^{-1},\]  and recalling $\sigma_{ij}^2=\frac{1}{n}$, we can expand $v_i$ as
\[v_i = m_{sc}^2(\frac{1}{n}\sum_jv_j - \Upsilon_i) - m_{sc}^3(\frac{1}{n}\sum_jv_j-\Upsilon_i)^2 + O(\frac{1}{n}\sum_jv_j-\Upsilon_i)^3.\]
We sum over the indices $i$ and divide by $n$.  Using the notation $[v]:= \frac{1}{n}\sum_iv_i$, we obtain
\[
[v] = m_{sc}^2(z)[v] - \frac{m_{sc}^2(z)}{n}\sum_i\Upsilon_i - \frac{m_{sc}^3(z)}{n}\sum_i([v]-\Upsilon_i)^2 + O((\Lambda_d + \Upsilon)^3). \]
In the above equation, we have denoted 
\[\Upsilon = \max_i |\Upsilon_i|.\]
From now on, we will omit the dependency on $z$.  The equation reduces to
\[ C_1[v] = C_2 \frac{1}{n}\sum_i\Upsilon_i +  C_3 \frac{1}{n}\sum_i([v]-\Upsilon_i)^2 + O((\Lambda_d + \Upsilon)^3),\]
where 
\begin{align*}
C_1(z)&\asymp \sqrt{\kappa +\eta}\\
\kappa &= |2-\Re z|,
\end{align*}
see \cite{EYY}, (3.13), and $C_2$, $C_3$ are bounded independently of $n$.

 are deterministic constants independent of $n$. We would like to prove the following bound:
\[ \mathbf{E}(|[v]|^2) \lesssim \frac{1}{n^2\eta^2}.\]
Taking squares, then expectations, and using Cauchy-Schwarz, we have
\begin{equation} \label{eqn:error}
\mathbf{E}(|[v]|^2) \lesssim \frac{1}{\kappa+\eta}\left(\frac{1}{n^2}\mathbf{E}\big|\sum_i\Upsilon_i\big|^2 + \frac{1}{n^2}\mathbf{E}\big|(\sum_i([v]-\Upsilon_i)^2\big|^2 + \mathbf{E}|O(\Lambda_d + \Upsilon)^3|^2\right).
\end{equation}
We begin with the first and second terms in (\ref{eqn:error}). The remainder term will be dealt with at the end.
From $\Upsilon_i = A_i + h_{ii} - Z_i$, we have
\[ \frac{1}{n^2}\big|\sum_i\Upsilon_i\big|^2 = \frac{1}{n^2}\big|\sum_iA_i + \sum_ih_{ii} - \sum_i Z_i\big|^2.\]
The right hand side is bounded up to a constant factor by
\[\frac{1}{n^2}\left(\big|\sum_iA_i\big|^2 + \big|\sum_ih_{ii}\big|^2 + \big|\sum_iZ_i\big|^2\right).\]
Thus, it will suffice to establish the three estimates
\begin{align*}
\mathbf{E}\big|\sum_ih_{ii}\big|^2 \lesssim \frac{\kappa + \eta}{\eta^2},\\
\mathbf{E}\big|\sum_iA_i\big|^2 \lesssim \frac{\kappa + \eta}{\eta^2},
\end{align*}
and
\begin{equation*}
\mathbf{E}\big|\sum_iZ_i\big|^2 \lesssim \frac{\kappa + \eta}{\eta^2}.
\end{equation*}
The first one is obvious, since $\mathbf{E}|h_{ii}|^2=\frac{1}{n}$. We proceed to prove the second bound.  Using the  statement concerning individual elements of the resolvent in Theorem \ref{thm:EYY}, we have
\begin{equation}\label{eq: Aibound}
\left|\frac{1}{n}\sum_{ij}\frac{G_{ij}G_{ji}}{G_{ii}}\right| \leq \left|\frac{m_{sc}(z)}{n}\sum_{ij}G_{ij}G_{ji}\right| + \frac{(\log n)^{4L}}{n\sqrt{n\eta}}O(\sum_{ij}|G_{ij}||G_{ji}|).\end{equation}
For the first term, we write out 
\[G_{ij} = \sum_m\frac{u_m(i)\overline{u_m(j)}}{\lambda_m-z},\] 
where $(u_m(i))_{1\le i\le n}$ is the eigenvector of $H$ corresponding to the eigenvalue $\lambda_m$. We expand the double sum using orthogonality of the $u_m$, and find that it is bounded by
\[\left|\frac{m_{sc}(z)}{n}\sum_m\frac{1}{(\lambda_m-z)^2}\right| = \left|m_{sc}(z)\frac{\mathrm{d}}{\mathrm{d}z} m(z)\right|.\]
By the strong local semicircle law and Cauchy's integral formula, this is bounded by $1/\sqrt{\eta}$, up to a constant factor for $\eta>n^{-1/10}$, with overwhelming probability. Squaring this, we obtain
\[\left|\frac{m_{sc}(z)}{n}\sum_{ij}G_{ij}G_{ji}\right|^2 \lesssim \frac{1}{\eta} \le \frac{\kappa +\eta}{\eta^2},\]
holding with overwhelming probability.
Since $|A_i|$ has a polynomial upper bound, we have obtained the desired estimate for the main term in \eqref{eq: Aibound}. For the error term, we use 
\[|G_{ij}| \lesssim \frac{1}{\sqrt{n\eta}}\] 
to obtain the bound, again using the fact that $\eta >  n^{-1/10}$.

It remains to bound the quantity 
\[\mathbf{E}\big|\sum_iZ_i\big|^2.\]
Since $\mathbf{E}Z_i = 0$, the variance of $[Z]$ is given by
\begin{equation}
\label{eq: Zsquared}
\frac{1}{n^2}\mathbf{E}|\sum_i^nZ_i|^2 = \frac{1}{n^2}\mathbf{E}\sum_{\alpha\neq\beta}\bar{Z_\alpha}Z_\beta + \frac{1}{n^2}\mathbf{E}\sum_\alpha|Z_\alpha|^2.
\end{equation}
We shall focus on the case where $\alpha = 1$ and $\beta = 2$.  Writing $G_{kl}^{(1)}$ as
\[G_{kl}^{(1)} = P_{kl}^{(12)}+P_{kl}^{(1)}.\]
where $P^{(1)}$ and $P^{(2)}$ (for $k,l \neq 1$) are defined as 
\[P_{kl}^{(12)}:=G_{kl}^{(12)}, P_{kl}^{(1)}:= \frac{G^{(1)}_{k2}G^{(1)}_{2l}}{G^{(1)}_{22}} \qquad \textrm{if } ~k,l\neq2.\]
\[
P_{kl}^{(12)}:=0, P_{kl}^{(1)}:=G_{kl}^{(1)}\qquad \textrm{if } k=2\textrm{ or }l=2.\]
With these notations and denoting $\mathbf{I}\mathbf{E}_i:= \mathbf{I}-\mathbf{E}_{\mathbf{a}_i}$, we can express
\[Z_1 = \mathbf{I}\mathbf{E}_1\mathbf{a}^1\cdot P^{(12)}\mathbf{a}^1 + \mathbf{I}\mathbf{E}_1\mathbf{a}^1\cdot P^{1}\mathbf{a}^1\]
and
\[ Z_2 = \mathbf{I}\mathbf{E}_2\mathbf{a}^2\cdot P^{(21)}\mathbf{a}^2 + \mathbf{I}\mathbf{E}_2\mathbf{a}^2\cdot P^{2}\mathbf{a}^2.\]

Using the previous expressions, we see that to estimate the right side of \eqref{eq: Zsquared}, it suffices to obtain the following bound:
\[\mathbf{E}|\mathbf{a}^1\cdot P^{(1)}\mathbf{a}^1|^2 \lesssim \frac{1}{(n\eta)^2}.\]
Follow the idea in section 8.1 of \cite{EYY2}, we write
\[
\mathbf{E}|\mathbf{a}^1\cdot P^{(1)}\mathbf{a}^1|^2 = \mathbf{E}\left|\sum_{k,l}\bar{a^1_k}\frac{G_{k2}^{(1)}G_{2l}^{(1)}}{G_{22}^{(1)}}a^1_l\right|^2.\]
Expanding, we have
\[\mathbf{E}\mathbf{E}_1\sum_{k,l,p,q}\overline{a_k^1}a_l^1a_p^1\overline{a_q^1} \frac{G_{k2}G_{2l}\overline{G_{p2}}\overline{G_{2q}}}{|G_{22}|^2}.\]
There are only a few cases where the first expectation is nonzero. One such case is $k=l$, $p=q$.  The other situations can be dealt with in similar fashion.  Taking the first expectation in, we have to bound
\[\frac{1}{n^2}\mathbf{E}\left(\sum_k \frac{G_{k2}G_{2k}}{|G_{22}|}\right)^2.\]
We expand
\[|G_{22}|= m_{sc}(z)+ O\left(\frac{(\log n)^{4L}}{\sqrt{n\eta}}\right)\]
with overwhelming probability.
Thus, the only important quantity is
\begin{equation}\label{eq: Gk2}
\frac{1}{n^2}\mathbf{E}\left(\sum_k G_{k2}G_{2k}\right)^2.
\end{equation}
Expanding the above sum, we obtain
\[\frac{1}{n^2}\mathbf{E}\left(\sum_m\frac{|u_m(2)|^2}{(\lambda_m-z)^2}\right)^2 = \frac{1}{n^2}\mathbf{E}\left(\frac{\mathrm{d}}{\mathrm{d}z}G_{22}(z)\right)^2\]
Using Theorem \cite{EYY} and Cauchy's estimate once more, we obtain
\[\left|\frac{1}{n^2}\mathbf{E}\left(\sum_k \frac{G_{k2}G_{2k}}{|G_{22}|}\right)^2\right|\lesssim \frac{1}{n^2\eta},\]
with overwhelming probability. Again, all quantities involved have polynomial bounds, so we have obtained the desired estimate.

Next, we direct our attention to the second term on the right hand side of equation (\ref{eqn:error}).  Using Theorem (\ref{thm:EYY}), we see that the term $\frac{1}{n^2}(\sum_i[v]^2)^2$ will be bounded by $n^{-4+\epsilon}\eta^{-4}$ and since $\eta > n^{-1/10}$, will in turn be bounded by $n^{-2}\eta^{-1}$ if $\eta>n^{-2/3}$.  So it remains to bound the term
\[\mathbf{E}\frac{1}{n^2}|\sum_i\Upsilon_i^2|^2 \lesssim \frac{1}{n^2}\left(\mathbf{E}|\sum_iA_i^2|^2 + \mathbf{E}|\sum_ih_{ii}^2|^2 + \mathbf{E}|\sum_iZ_i^2|^2\right).\]
Again by Theorem (\ref{thm:EYY}) and equation (\ref{eqn:A}), the first term on the right hand side is harmless.  The second term is trivial to bound so we focus on the third term.
Applying Cauchy-Schwarz twice, we have
\begin{equation}
\frac{1}{n^2}\mathbf{E}\left|\sum_iZ_i^2\right|^2 \leq \frac{1}{n}\sum_i\mathbf{E}|Z_i|^4.
\end{equation}
The desired bound then follows immediately from lemma \ref{eyylemma} below, the definition of $Z$, and the estimates
\[\left(\sum_{i}|G_{ii}^{(j)}|^2\right)^{1/2} \lesssim \sqrt{n}.\]
By (3.39), (3.41) in \cite{EYY}, we have
\[\left(\sum_{i\neq j} |G^{(k)}_{ij}|^2\right)^{1/2} \leq \left(\sum_m \frac{1}{|\lambda_m^{(k)}-z|^2}\right)^{1/2} \lesssim (\Im m_{sc}(z))^{1/2}n^{-1/2}\eta^{-1/2},\]
holding with overwhelming probability.
The proof of the next lemma can be found in Appendix B of \cite{EYY3}.
\begin{lemma}\label{eyylemma}
  Let $a_i$ $(1\leq i \leq n)$ be $n$ independent random complex variables with mean zero, variance $\sigma^2$ and having uniform subexponential decay.  Let $B_{ij} \in \mathbb{C}$ $(1\leq i \leq n)$.  Then we have that
\begin{equation*}
\mathbb{P}\{|\sum_{i=1}^n\bar{a}_iB_{ii}a_i - \sum_{i=1}^n\sigma^2B_{ii}| \geq D\sigma^2(\sum_{i=1}^n|B_{ii}|^2)^{1/2}\} \leq C\exp(-cD^{\frac{1}{1+\alpha}})
\end{equation*}
and
\[
  \mathbb{P}\{|\sum_{i\neq j}\bar{a}_iB_{ij}a_j| \geq D\sigma^2(\sum_{i\neq j}|B_{ij}|^2)^{1/2}\} \leq C\exp(-cD^{\frac{1}{2(1+\alpha)}}).
\]
\end{lemma}
Lastly, the last term in equation (\ref{eqn:error}) can be bounded by $n^{-3/2+\epsilon}\eta^{-3/2}$ because both $\Lambda_d$ and $\Upsilon$ are bounded by $n^{-1/2+\epsilon}\eta^{-1/2}$ with overwhelming probability by Theorem \ref{thm:EYY}.  For $\eta > n^{-1/10}$, we have $n^{-3/2+\epsilon}\eta^{-3/2} < n^{-1/2}\eta^{-1/2}$ and this concludes the proof of (\ref{eqn:varbound}).
\end{proof}
We are now ready to prove Theorem \ref{thm2}. Write $\varphi= \chi \varphi+ (1-\chi) \varphi$, where $\chi$ is a cut-off function equal to $1$ in $[-3,3]$ and equal to $0$ in $[-4,4]^c$.
Then we have
\begin{align*}
\mathbf{Var}(\mathcal{N}_n[\varphi])&=\mathbf{Var}(\mathcal{N}_n[\chi \varphi]+\mathcal{N}_n[(1-\chi)\varphi])\\
& \le 2\mathbf{Var}(\mathcal{N}_n[\chi \varphi]) + 2\mathbf{Var}(\mathcal{N}_n[(1-\chi)\varphi]).
\end{align*}
We bound the second term by the second moment $\mathbf{E}|\mathcal{N}_n[(1-\chi)\varphi|^2$, which can be estimated as follows:
\[\mathbf{P}(|\sum_i \left((1-\chi)\varphi\right)(\lambda_i)|>0)\le \mathbf{P}(\lambda_1 > 4)+\mathbf{P}(\lambda_n <-4).\]
The last quantity is bounded by $e^{-n^c}$ (\cite{EYY3}, Lemma 7.2).
Thus
\begin{align*}
\mathbf{E}|\mathcal{N}_n[(1-\chi)\varphi]|^2 &= 2\int_0^{n\|\varphi\|_{L^\infty}} x\mathbf{P}( |\sum_i \varphi(\lambda_i)|\ge x)\,\mathrm{d}x\\
&\lesssim n^2\|\varphi\|^2_{L^\infty}\cdot \mathbf{P}( |\sum_i \varphi(\lambda_i)|> 0)\\
&= n^2e^{-n^c}\|\varphi\|^2_{L^\infty} \lesssim \|\varphi\|^2_{H^{1+\epsilon}}.
\end{align*}
The final inequality follows from the Sobolev inequality.
For the first term above, use the decomposition (\ref{eq: lpvariancebound}) for $\varphi$, but not for $\chi$:
\begin{equation}
\label{eq: crucialCS}
\mathbf{Var}(\mathcal{N}_n[\chi \varphi])\le \sum_{k,j} \int\int |g_k||g_j| \mathbf{Cov}(\mathcal{N}_n[\chi(\cdot)P_{2^{-k}}(t-\cdot)],\mathcal{N}_n[\chi(\cdot)P_{2^{-j}}(s-\cdot)])\mathrm{d}t\,\mathrm{d}s.
\end{equation}
Apply Cauchy-Schwarz to the covariance on the right side of the previous equation. Since the integrals are now decoupled, we concentrate on bounding:
\[\sum_{k} \int |g_k| \mathbf{Var}(\mathcal{N}_n[\chi(\cdot) P_{2^{-k}}(t-\cdot)])^{1/2} \,\mathrm{d}t.\]
Apply Cauchy-Schwarz again to the integral
\begin{equation} \label{eq: halfhalf} \sum_{k} \int 2^{(1+\epsilon)k} |g_k|  2^{-(1+\epsilon)k} \mathbf{Var}(\mathcal{N}_n[\chi(\cdot) P_{2^{-k}}(t-\cdot)])^{1/2} \,\mathrm{d}t\end{equation}
and obtain
\[\sum_k \left(\int 2^{(2+2\epsilon)k} |g_k|^2 \,\mathrm{d}t\right)^{1/2} \left( \int 2^{-(2+2\epsilon)k} \mathbf{Var}(\mathcal{N}_n[\chi(\cdot) P_{2^{-k}}(t-\cdot)]) \,\mathrm{d}t \right)^{1/2}.\]
We are left to examine the quantity:
\[\int_{\mathbb{R}} \mathbf{Var}(\mathcal{N}_n[\chi(\cdot)P_{2^{-k}}(t-\cdot)])\mathrm{d}t.\]
Split the integral:
\begin{equation} \label{eq: splitintegral}\int_{-5}^5 \mathbf{Var}(\mathcal{N}_n[\chi(\cdot)P_{2^{-k}}(t-\cdot)]) \,\mathrm{d}t + \int_{[-5,5]^c} \mathbf{Var}(\mathcal{N}_n[\chi(\cdot)P_{2^{-k}}(t-\cdot)]) \,\mathrm{d}t.\end{equation}
The integrand in the second term is the variance of a smooth function with support in $[-3,3]$, and such that
\[\|\chi P_{2^{-k}}(t-\cdot)\|_{C^{2}([-5,5])}\lesssim \frac{C}{(t-2)^2}.\]
By M. Shcherbina's variance bound \cite{shcherbina} for linear statistics of $H^{3/2}$ functions, this term is thus integrable.
The first term in \eqref{eq: splitintegral} can be written as a sum:
\begin{multline*}
\mathbf{Var}(\mathcal{N}_n[\chi(\cdot)P_{2^{-k}}(t-\cdot)]; |\lambda_1|, |\lambda_n|\le 2) \\+ \mathbf{Var}(\mathcal{N}_n[\chi(\cdot)P_{2^{-k}}(t-\cdot)]; \{|\lambda_1|, |\lambda_n|\le 2\}^c).
\end{multline*}
Since $\chi = 1$ on $[-2,2]$:
\[\mathbf{Var}(\mathcal{N}_n[\chi(\cdot)P_{2^{-k}}(t-\cdot)]; |\lambda_1|, |\lambda_n|\le 2) \le \mathbf{Var}(\mathcal{N}_n[P_{2^{-k}}(t-\cdot)]),\]
which is bounded by $2^{2k+\epsilon}$ by (\ref{eqn:varbound}). On other hand,
\[\mathbf{Var}(\mathcal{N}_n[\chi(\cdot)P_{2^{-k}}(t-\cdot)]; \{|\lambda_1|, |\lambda_n|\le 2\}^c) \le 2^{2k}n^2e^{-n^c}.\]
Thus, (\ref{eq: halfhalf}) is bounded up to a constant factor by:
\[\sum_k \left(\int 2^{(2+2\epsilon)k} |g_k|^2 \,\mathrm{d}t\right)^{1/2}.\]
Choosing $\epsilon$ suitably, a final application of Cauchy-Schwarz and Theorem \ref{thm: lpcharacterization} end the proof of the bound (\ref{eq: h1variance}).  The central limit theorem then follows from the discussion in Section \ref{sec: approx}.
\begin{remark}
A proof along the same lines as above yields a variance bound and Gaussian fluctuations of linear statistics for functions in the Besov space $B^{1+}_{p,\infty}$ (see \cite{bahourietal}, Chapter 2 for a definition) for all $p \geq 1$.
\end{remark}
\section{Variance bound and CLT for $C^{\frac{1}{2}+}$ and $\dot{H}^{1/2}$: GUE case}
\label{sec: chalfgue}
Our starting point is the following representation for the variance, based on the determinant kernel representation of the 2-point function:
\begin{equation}\label{eq: variancebound}
\mathbf{Var}\left(\mathcal{N}_n[\varphi]\right) = \frac{1}{2}\iint_{\mathbb{R}^2}|\varphi(x)-\varphi(y)|^2\,(K_n(x,y))^2\,\mathrm{d}y\mathrm{d}x,
\end{equation}
with the kernel $K_n$ being defined in terms of Hermite functions by the Christoffel-Darboux formula:
\[K_n(x,y) =\sum_{j=0}^{n-1}\psi_j(x)\psi_j(y) = \frac{\psi_n(x)\psi_{n-1}(y)-\psi_n(y)\psi_{n-1}(x)}{x-y}.\]
The Hermite functions are defined by:
\begin{align*}
H_k(x)&=(-1)^ke^{-x^2}\frac{\mathrm{d}^k}{\mathrm{d}x^k}e^{x^2}\\
\tilde{\psi}_k(x) &= \frac{H_k(x)}{(2\pi)^{1/4}(k!)^{1/2}}e^{-x^2/2}\\
\psi_k(x) &= \left(\frac{n}{2}\right)^{1/4}\tilde{\psi}_k\left(\frac{n^{1/2}x}{2}\right).
\end{align*}
Let us recall how formula (\ref{eq: variancebound}) is derived. We can write the variance as:
\begin{align*}
\mathbf{Var}\left(\mathcal{N}_n[\varphi]\right)&=\mathbf{E}\mathcal{N}^2_n[\varphi]-(\mathbf{E}\mathcal{N}_n[\varphi])^2\\
&=\mathbf{E}\sum_{i\neq j}\varphi(\lambda_i)\varphi(\lambda_j)+\mathbf{E}\sum_i\varphi^2(\lambda_i)\\
&\quad -\sum_{i,j}\mathbf{E}\varphi(\lambda_i)\mathbf{E}\varphi(\lambda_j).
\end{align*}
In the second equality we have expanded the definition of $\mathcal{N}_n[\varphi]$. By definition
\begin{equation} \label{eqn: correlationkernel}
\mathbf{E}\sum_{i\neq j}\varphi(\lambda_i)\varphi(\lambda_j)= n(n-1)\iint\varphi(\lambda)\varphi(\mu)\,p_n^{(2)}(\lambda,\mu)\,\mathrm{d}\lambda\mathrm{d}\mu.
\end{equation}
\begin{align} \label{eqn: squarekernel}
\mathbf{E}\sum_i \varphi(\lambda_i)^2 &= n\int\varphi^2(\lambda)\,p_n^{(1)}(\lambda)\,\mathrm{d}\lambda\\
&= \iint \varphi^2(\lambda)\,(n^2p_n^{(1)}(\lambda)p_n^{(1)}(\mu)-n(n-1)p_n^{(2)}(\lambda,\mu))\mathrm{d}\mu\mathrm{d}\lambda \nonumber \\
&= \frac{1}{2} \iint (\varphi^2(\lambda)+\varphi^2(\mu) )\,(n^2p_n^{(1)}(\lambda)p_n^{(1)}(\mu)-n(n-1)p_n^{(2)}(\lambda,\mu))\mathrm{d}\mu\mathrm{d}\lambda,\nonumber
\end{align}
\begin{equation}\label{eqn: doubletrace}
\sum_{i,j}\mathbf{E}\varphi(\lambda_i)\mathbf{E}\varphi(\lambda_j) = n^2\iint\varphi(\mu)\varphi(\lambda)\,p_n^{(1)}(\lambda)p_n^{(1)}(\mu)\,\mathrm{d}\mu\mathrm{d}\lambda.
\end{equation}
In the above, $p_n^{(k)}$ denotes the $k$-point correlation function of the GUE eigenvalues. By the Dyson-Gaudin-Mehta computation, we have:
\begin{align*}
p_n^{(1)}(\lambda)&= K_n(\lambda,\lambda)\\
p_n^{(2)}(\lambda,\mu)&=K_n(\lambda,\lambda)K_n(\mu,\mu)-K_n(\lambda,\mu)K_n(\mu,\lambda)
\end{align*}
Using symmetry of $K_n$, i.e. $K_n(\lambda,\mu)=K_n(\mu,\lambda)$, we arrive at the variance formula \eqref{eq: variancebound}.

Let us return to the variance bound. M. Shcherbina and L. Pastur \cite{shcherbinapastur}, Theorem 5.2.7, show that 
\begin{equation}\label{eq: mainbound}
\mathbf{Var}\left(\mathcal{N}_n[\varphi]\right)= \frac{1}{4\pi^2}\iint_{[-2,2]^2} \frac{(\varphi(x)-\varphi(y))^2}{(x-y)^2}\frac{4-xy}{\sqrt{4-x^2}\sqrt{4-y^2}} \mathrm{d}x\mathrm{d}y +o(1)
\end{equation}
for Lipschitz functions $\varphi$. By carrying out the computation carefully, we will see that the result extends to $C^{1/2+}$ functions, and the convergence is uniform over bounded subsets of $\|\varphi\|_{C^{1/2+\epsilon}}$. 

Let us first show that the main term in (\ref{eq: mainbound}) can be controlled by $\|\varphi\|_{C^{\frac{1}{2}+\epsilon}}$. The first step is to make use of the $C^{1/2+\epsilon}$ hypothesis. Let $\epsilon>0$ such that $\varphi \in C^{1/2+\epsilon}$, then
\begin{multline*}
\iint_{[-2,2]^2} \frac{(\varphi(x)-\varphi(y))^2}{(x-y)^2}\frac{4-xy}{\sqrt{4-x^2}\sqrt{4-y^2}} \mathrm{d}x\mathrm{d}y\\ \le [\varphi]^2_{C^{1/2+\epsilon}}\iint_{[-2,2]^2}\frac{1}{|x-y|^{1-2\epsilon}} \frac{4-xy}{\sqrt{4-x^2}\sqrt{4-y^2}} \,\mathrm{d}x\mathrm{d}y.
\end{multline*}
Thus we just need to check that the integral is finite. It is clear that the integrand is locally integrable away from the points $(-2,-2)$ and $(2,2)$. Also, the integrand is positive. By symmetry, it is sufficient to check finiteness of the integral:
\begin{multline*}
\int_0^2 \int_0^x \frac{1}{|x-y|^{1-2\epsilon}} \frac{4-xy}{\sqrt{4-x^2}\sqrt{4-y^2}} \,\mathrm{d}y\mathrm{d}x\\
\le \frac{1}{2} \int_0^2 \int_0^x \frac{1}{|x-y|^{1-2\epsilon}} \frac{4-xy}{\sqrt{2-x}\sqrt{2-y}} \,\mathrm{d}y
\mathrm{d}x
\end{multline*}
Writing $4-xy= 2(2-x)+(2-y)x$, the last integral can be bounded by:
\[2\int_0^2 \int_0^x \frac{1}{|x-y|^{1-2\epsilon}} \frac{\sqrt{2-x}}{\sqrt{2-y}} \,\mathrm{d}y
\mathrm{d}x + \int_0^2 \int_0^x \frac{1}{|x-y|^{1-2\epsilon}} \frac{x\sqrt{2-y}}{\sqrt{2-x}} \,\mathrm{d}y
\mathrm{d}x.  \]
Since $\sqrt{2-y}>\sqrt{2-x}$ on the domain of integration, the first integral is convergent; by considering the cases $|x-y|\le |2-x|^{\frac{1}{2}}$ and $|x-y|\ge|2-x|^{\frac{1}{2}}$, we see that the second integral is convergent as well.

We now turn to the computation of (\ref{eq: mainbound}). The main tool will be the Plancherel-Rotach asymptotics for the Hermite polynomials \cite{szego}, Chapter 8. We begin by showing that we can replace the integral in (\ref{eq: variancebound}) by an integral over the square 
\[\max(|x|,|y|)\le 2+ n^{-\nu}\] 
for some $0<\nu<2/3$ to be selected later. This follows by the simple bound \cite{shcherbinapastur} (5.2.8)
\[K_n(2+sn^{-2/3},2+sn^{-2/3})\le Cs^{-1}n^{2/3}e^{-cs^{3/2}}.\]
This bound in turn follows from Plancherel-Rotach asymptotics in the transition region. Applying it, we find that
\[\mathbf{Var}\left(\mathcal{N}_n[\varphi]\right)=\frac{1}{2}\iint_{|x|,|y|\le 2+n^{-\nu}} |\varphi(x)-\varphi(y)|^2\,(K_n(x,y))^2\,\mathrm{d}y\mathrm{d}x + \|\varphi\|^2_\infty O(n^{-\infty}).\]
We will now further replace the integral over the square $\{|x|,|y|\le 2+n^{-\nu}\}$ by an integral over
$\{|x|,|y|\le 2-n^{-\nu}\}$.  The difference between these two squares is contained in the union of four rectangles.
Without loss of generality, we can consider the integral over the ``bottom'' rectangle
\[R_n = [-2-n^{-\nu},2+n^{-\nu}]\times [-2-n^{-\nu},-2+n^{-\nu}].\] 
We write:
\begin{multline*}
\iint_{R_n}\frac{(\varphi(x)-\varphi(y))^2}{(x-y)^2} (\psi_n(x)\psi_{n-1}(y)-\psi_n(y)\psi_{n-1}(x))^2\mathrm{d}x\mathrm{d}y\\
\le [\varphi]^2_{C^{1/2+\epsilon}}\int_{-2-n^{-\nu}}^{2+n^{-\nu}}\int_{-2-n^{-\nu}}^{-2+n^{-\nu}} \frac{1}{|x-y|^{1-2\epsilon}} (\psi_n(x)\psi_{n-1}(y)-\psi_n(y)\psi_{n-1}(x))^2\mathrm{d}x\mathrm{d}y.
\end{multline*}
The last integral is bounded up to a constant factor by:
\[\int_{-2-n^{-\nu}}^{2+n^{-\nu}}\int_{-2-n^{-\nu}}^{-2+n^{-\nu}} \frac{1}{|x-y|^{1-2\epsilon}} \left((\psi_n(x)\psi_{n-1}(y))^2+(\psi_n(y)\psi_{n-1}(x))^2\right)\,\mathrm{d}x\mathrm{d}y.\]
By symmetry (because $\psi_n$ and $\psi_{n-1}$ have similar behaviour as far as we are concerned), it suffices to deal with:
\begin{equation}
\label{eq: intoverR}
  \int_{-2-n^{-\nu}}^{2+n^{-\nu}}(\psi_n(x))^2  \int_{-2-n^{-\nu}}^{-2+n^{-\nu}}(\psi_{n-1}(y))^2 \frac{1}{|x-y|^{1-2\epsilon}}\,\mathrm{d}y\mathrm{d}x.
  \end{equation}
Using the bound $|\tilde{\psi}_n|\le Cn^{-\frac{1}{12}}$ for all $n$ \cite{szego}, Theorem 8.22.29 we have $(\psi_{n-1}(y))^2\le Cn^{-1/6}n^{1/2}=Cn^{1/3}$. Inserting this bound in the inner integral and computing, we find:
\begin{align}\label{eq: innerintbound}
\int_{-2-n^{-\nu}}^{-2+n^{-\nu}} (\psi_{n-1}(y))^2 \frac{1}{|x-y|^{1-2\epsilon}}\,\mathrm{d}y
&\le C_{\epsilon}n^{1/3}\frac{n^{-\nu}}{|x+2|^{1-2\epsilon}}\mathbf{1}_{[-2-2n^{-\nu},-2+2n^{-\nu}]^c}(x)\\&\quad+C_\epsilon n^{1/3}\mathbf{1}_{[-2-2n^{-\nu},-2+2n^{-\nu}]}(x) n^{-2\epsilon \cdot \nu}.\nonumber
\end{align}
The last estimate is obtained by direct computation. For the first term on the right, we have used the Taylor expansion of $(2+n^{-\nu}/\alpha)^{2\epsilon}$ with $\alpha=x+2$, valid in the range $n^{-\nu}\le \frac{1}{2}|\alpha|$. We insert this bound in the $x$ integral in $(\ref{eq: intoverR})$. In the region $-2+O(n^{-\nu})$, corresponding on the right side of (\ref{eq: innerintbound}), we use the bound $(\psi_n(x))^2\le Cn^{1/3}$, leading to a bound of order
\[C_\epsilon n^{2/3}n^{-\nu-2\epsilon\cdot \nu}.\]
Choosing, as we may,
\[\frac{2}{3}\frac{1}{1+2\epsilon} < \nu  <\frac{2}{3},\]
the contribution from the second term to (\ref{eq: intoverR}) is $O(n^{-c})$ for $c>0$. The first term on the right of (\ref{eq: innerintbound}) leads to a contribution bounded by:

\begin{multline} \label{eq: integral}
n^{1/3}n^{-\nu} \int_{-2+2n^{-\nu}}^{2+n^\nu} (\psi_n(x))^2 {|x+2|^{-1+2\epsilon}}\,\mathrm{d}x \\= n^{1/3-\nu}\times \left( \int_{-2+2n^{-\nu}}^{0} + \int_0^{2+n^{-\nu}}\right) \ldots .
\end{multline}
For the first integral, notice that since $\nu<2/3$, we are far clear from the transition region of the rescaled Hermite functions $\psi_n$, and so we can apply the asymptotics  in \cite{erdelyi} (see also \cite{shcherbinapastur} (5.1.9)), which show that $\tilde{\psi}_n$, $\tilde{\psi}_{n-1}$ are bounded by
\[\frac{C}{\sqrt{2-x^2}}\]
in the region of integration.
Thus the first term on the right side of (\ref{eq: integral}) is bounded by:
\[Cn^{1/3}n^{-\nu}\int_{-2+2n^{-\nu}}^0\frac{1}{\sqrt{2-x^2}}\frac{1}{|x+2|^{1-2\epsilon}}\,\mathrm{d}x=O(n^{1/3-\epsilon\nu-\nu/2}).\]
For the second integral on the right side of (\ref{eq: integral}), notice that $|x+2|$ is bounded below on the range of integration, and that on the other hand, the total integral of $(\psi_n(x))^2$ is equal to $1$. Thus the contribution from this integral is $O(n^{1/3-\nu})$.

We have so far established the asymptotic behaviour:
\[\mathbf{Var}\left(\mathcal{N}_n[\varphi]\right)=\iint_{|x|,|y|\le 2-n^{-\nu}} |\varphi(x)-\varphi(y)|^2\,(K_n(x,y))^2\,\mathrm{d}y\mathrm{d}x +C\|\varphi\|^2_{C^{1/2+\epsilon}}n^{-c}.\]
We now wish to replace $n^{-\nu}$ (whose smallness has been useful so far) in the integral range by a fixed $\delta>0$. We will bound the difference, which is the sum of four integrals of the form
\[\int_{2-\delta \le |x| \le 2-n^{-\nu}}\int_{|y| \le 2-n^{\nu}}\]
in terms of $\delta$, independently of $n$. Indeed, as explained above, since $\nu<2/3$, the Hermite functions are bounded in the region of interest. Inserting these $L^\infty$ bounds in the integral, and extending the upper limits of integration from $2-n^{-\nu}$ to $2$, we find that the difference is estimated by
\begin{equation}
\label{eq: delta}
C[\varphi]^2_{C^{1/2+\epsilon}}\int_{2-\delta \le |x| \le 2}\int_{|y| \le 2}\frac{1}{|x-y|^{1-2\epsilon}}\,\mathrm{d}x\mathrm{d}y= [\varphi]^2_{C^{1/2+\epsilon}} \cdot O(\delta).\end{equation}

Thus we are now concerned with computing the quantity:
\begin{equation}
\label{eq: intoverdelta}
\iint_{|x|,|y|\le 2-\delta} |\varphi(x)-\varphi(y)|^2\,(K_n(x,y))^2\,\mathrm{d}y\mathrm{d}x
\end{equation}
for $\delta>0$. We need to show that it matches the first term in (\ref{eq: mainbound}) with an error uniform in 
$\|\varphi\|_{C^{1/2+\epsilon}}\le K$ for any $K>0$.  Expanding $(K_n)^2$, we find:
\[(\psi_n(x))^2(\psi_{n-1}(y))^2+(\psi_n(y))^2\psi_{n-1}(x))^2-2\psi_n(x)\psi_n(y)\psi_{n-1}(x)\psi_{n-1}(y).\]
By symmetry, it suffices to consider the contributions of the first and last terms.
Using the Plancherel-Rotach asymptotics and trigonometric identities, we have:
\begin{align*}
(\psi_n(x))^2&=\frac{1}{\sqrt{2-x^2}}(\cos^2(n\alpha(\theta)-\theta/2-\pi/4)+O(n^{-1}))\\
&=\frac{1}{\sqrt{2-x^2}}(1+\cos(2n\alpha(\theta)-\theta-\pi/2))+O(n^{-1}),
\end{align*}
uniformly in $|x|\le 2-\delta$, where 
\begin{align*}
x&=\cos \theta\\
\alpha(\theta)&=\theta-\sin(2\theta)/2
\end{align*}
Similarly
\[\psi_n(x)\psi_{n-1}(x)=\frac{1}{\sqrt{2-x^2}}(\cos(2\theta)+\cos(2n\alpha(\theta)-\theta-\pi/2))+O(n^{-1}).\]
We will now simply insert these expansions into the double integral over $|x|,|y|\le 2-\delta$, and use the $C^\alpha$ smoothness of the integrand (as opposed to the Riemann-Lebesgue lemma for $L^1$ as in Pastur and Shcherbina's computation) to obtain a uniform rate of decay (for each $\delta>0$ and uniform in $\|\varphi\|_{C^{1/2+\epsilon}}\le K$). 

We start with the integrals resulting from the square terms $(\psi_n(x))^2(\psi_{n-1}(x))^2$ and\\ $(\psi_n(y))^2(\psi_{n-1}(y))^2$:
\begin{multline*}\iint_{\{|x|,|y|\le 2-\delta\}} \frac{1}{\sqrt{4-x^2}\sqrt{4-y^2}}\frac{(\varphi(x)-\varphi(y))^2}{(x-y)^2} (1+\cos(2n\alpha(\theta(x))-\theta(x)-\pi/2)) \\ \times  (1+\cos(2n\alpha(\theta(y))-\theta(y)-\pi/2)) \,\mathrm{d}x\mathrm{d}y +[\varphi]_{C^{1/2+\epsilon}}O_\delta(n^{-1}).
\end{multline*}
The $O_\delta(\cdot)$ term is justified by $|\cos(\cdot)|\le 1$ and by substituting $|x-y|^{-1+2\epsilon}$ for the ratio $(\varphi(x)-\varphi(y))^2/(x-y)^2$.
Similarly, the integrals resulting from expanding 
\[-2\psi_n(x)\psi_n(y)\psi_{n-1}(x)\psi_{n}(y)\]
 gives a contribution of
\begin{multline*}-\iint_{\{|x|,|y|\le 2-\delta\}} \frac{1}{\sqrt{4-x^2}\sqrt{4-y^2}}\frac{(\varphi(x)-\varphi(y))^2}{(x-y)^2} (x+\cos(2n\alpha(\theta(x))-\theta(x)-\pi/2)) \\ \times  (y+\cos(2n\alpha(\theta(y))-\theta(y)-\pi/2)) \,\mathrm{d}x\mathrm{d}y +[\varphi]_{C^{1/2+\epsilon}}O_\delta(n^{-1}).
\end{multline*}
Adding the main term in the last two displayed equations, we find
\[\iint_{\{|x|,|y|\le 2-\delta\}} \frac{4-xy}{\sqrt{4-x^2}\sqrt{4-y^2}}\frac{(\varphi(x)-\varphi(y))^2}{(x-y)^2}\mathrm{d}x\mathrm{d}y,\]
as expected. It remains to show that the oscillatory integrals decay uniformly for $\varphi$ in a bounded set of $C^{1/2+\epsilon}$ and each $\delta>0$.
For this, it will suffice to deal with integrals of the form
\begin{equation} \label{eq: oscillation} \int_{-2+\delta}^{2-\delta}\cos(2n\alpha(\theta(x)))\frac{(\varphi(x)-\varphi(y))^2}{(x-y)^2}\frac{\mathrm{d}x}{\sqrt{4-x^2}}.\end{equation}
$\alpha$ and $x$ are related by the change of variables
\[\frac{\mathrm{d}\alpha(x)}{\mathrm{d}x}=-2\sin\theta(x)=-(4-x^2)^{1/2},\] 
non singular away from $x=0$ and $x=\pm 2$.
We treat (\ref{eq: oscillation}) as an oscillatory integral with non-stationary phase and H\"older continuous amplitude. The guiding model is the integral
\begin{equation}
\label{eq: model}
I(t)=\int_a^b e^{i\phi(x)t}A(x)\,\mathrm{d}x
\end{equation}
with $|\phi'|> \varepsilon$  and an amplitude $A\in C^\alpha$. This can be dealt with by writing:
\begin{align*}
\int_a^b e^{i\phi(x)t}a\left(x+\frac{\pi}{t\phi'(x)}\right)\,\mathrm{d}x &= \int_{a+O(1/t)}^{b+O(1/t)}  e^{i\phi\left(x-\frac{\pi}{t\phi'(x)}\right)t}A(x)\,\mathrm{d}x+O_{\|a\|_\infty,\epsilon}(1/t)\\
&=\int_a^be^{i(t\phi(x)+\pi)}e^{iO_{\varepsilon}(1/t^2)} A(x)\,\mathrm{d}x +O_{\|a\|_\infty,\epsilon}(1/t)\\
&= -I(t)+O(1/t).
\end{align*}
Thus we can write
\[2|I(t)| = \left|\int_a^b e^{i\phi(x)t}\left(A(x)-A(x+\pi/(t\phi'(x)))\right)\,\mathrm{d}x\right|+O(1/t).\]
By the $\alpha$-H\"older condition, the main term in the last equation is bounded by $C\|a\|_{C^\alpha}\cdot t^{-\alpha}$ if $a,b<\infty$. A similar analysis around around points $x$ where $\phi'(x)=0$ but $\phi''(x)>\epsilon$ shows that the contribution to the integral from a neighbourhood of that point is $C\|a\|_{C^\alpha}t^{-\alpha/2}$.

We will apply the same strategy to (\ref{eq: oscillation}), but we have to deal with the fact that the amplitude factor (seen as a function of $x$) becomes singular in a neighbourhood of $x=y$. Away from $y$, the algebra property of $C^\alpha$ ensures that the $C^\alpha$ norm of
\[\frac{(\varphi(x)-\varphi(y))^2}{(x-y)^2}\frac{1}{\sqrt{4-x^2}}\]
has a bound depending on $\delta$, $[\varphi]_{C^{1/2+\epsilon}}$ and the distance to $y$. Let us roughly calculate the dependence on the distance to $y$, i.e., we are interested in giving an estimate for the local norm
\[\|\cdot\|_{C^{1/2+\epsilon}(I)} =\|g_I \cdot \|_{C^{1/2+\epsilon}},\]
where $I\subset(-2+\delta,2-\delta)$ and $\operatorname{dist}(I,y)=s$, $g$ is a $C^\infty$ function supported on 
\[\{\operatorname{dist}(I,x)<s/10\}\] and $g\equiv1$ on $I$. The function $x\mapsto (\varphi(x)-\varphi(y))^2$ inherits its H\"older norm from $\varphi$, since:
\[|(\varphi(x)-\varphi(y))^2-(\varphi(x')-\varphi(y))^2| \le |\varphi(x)-\varphi(x')|\cdot 4\|\varphi\|_\infty.\]
We may ignore the factor $1/\sqrt{4-x^2}$, since its $C^1$ norm is bounded for $|x|<2-\delta$. To estimate the $C^\alpha$ norm of $1/(x-y)^2$, we simply interpolate between the $L^\infty$ and Lipschitz norms. For $x,y\in I+B(0,s)$:
\begin{align*}\frac{|f(x)-f(y)|}{|x-y|^\alpha}&\le |f(x)-f(y)|^{1-\alpha}\frac{|f(x)-f(y)|^\alpha}{|x-y|^\alpha}\\
&\le C\|f\|_{\infty}^{1-\alpha}\|f\|^\alpha_{C^\alpha}.
\end{align*}
If $\operatorname{dist}(I,x)\gtrsim s$, we thus have
\begin{align*}
\|1/(x-y)^2\|_{L^\infty(I)} &\lesssim s^{-2},\\
\|1/(x-y)^2\|_{C^1(I)} &\lesssim \|1/|x-y|^3\|_{L^\infty(I)} \lesssim s^{-3},\\
\|1/(x-y)^2\|_{C^{1/2+\epsilon}} &\lesssim s^{-5/2+\epsilon}.
\end{align*}
On the other hand, away from $x=\pm 2$, only the squared factor is singular, and so we can bound the integral by
\[C_\delta [\varphi]^2_{C^{1/2+\epsilon}}\int_{|x-y|<s, |x|<2}\frac{1}{|x-y|^{1-2\epsilon}}\,\mathrm{d}x< C_\delta s^{2\epsilon}.\]
It is now clear how to proceed: we split the integral (\ref{eq: oscillation}) according to the distance to $y$:
\[\int_{|x|<2-\delta, |x-y|<n^{-a}}  + \int_{|x|<2-\delta, |x-y|>n^{-a}} \cdots, \]
with $a>0$ a parameter to be adjusted later.
The contribution from the first term is bounded by $C_\delta [\varphi]_{C^{1/2+\epsilon}}^2n^{-a2\epsilon}$, as just explained. For the second part, we can (almost) repeat the computations leading  to the estimate for the model integral (\ref{eq: model}), using the estimate $C_\delta [\varphi]_{C^{1/2+\epsilon}}n^{a(5/2-\epsilon)}$ for the $1/2+\epsilon$ H\"older norm of the amplitude. All in all, we are lead to an estimate of the order
\[C_\delta [\varphi]_{C^{1/2+\epsilon}}n^{a(5/2-\epsilon)} \times n^{-1/4-\epsilon/2}.\]
However, we have ignored the fact that our $L^\infty$ norm estimates do not allow us to bound the error at the endpoints resulting from the translation of the variable $x$, i.e. the integral over
\[[y-n^{-a}-Cn^{-1},y-n^{-a}].\]
If instead of using the $L^\infty$ bound, we replace the integrand by $|x-y|^{-1+2\epsilon}$ once again, we obtain an error of order $[\varphi]^2_{C^{1/2+\epsilon}} n^{-2 \epsilon}$. Choosing $a$ appropriately ($a=\frac{1/2-\epsilon}{5-2\epsilon}$), we find that the model oscillatory integral (\ref{eq: oscillation}) decays at a power rate uniformly for $\varphi$ in bounded subsets of $C^{1/2+\epsilon}$ and uniformly in $|y|<2-\delta$. We can use the scheme above to deal with all the oscillatory terms we encounter, and thus we have established that
\begin{multline*}\mathbf{Var}(\mathcal{N}_n[\varphi]) = \iint_{\{|x|,|y|\le 2-\delta\}} \frac{4-xy}{\sqrt{4-x^2}\sqrt{4-y^2}}\frac{(\varphi(x)-\varphi(y))^2}{(x-y)^2}\mathrm{d}x\mathrm{d}y\\ + [\varphi]^2_{C^{1/2+\epsilon}}O(\delta) + O_{\|\varphi\|_{C^{1/2+\epsilon},\delta}}(n^{-\kappa})\end{multline*}
for some $\kappa>0$.
 
Given $K$ such that $\|\varphi\|_{C^{1/2+\epsilon}}\le K$, we can choose $\delta(K)$ to make the error in (\ref{eq: delta}) smaller than $\epsilon'$, and then choose $n(\delta)$ so as to make (\ref{eq: intoverdelta}) $\epsilon'$-close to the main term in the asymptotic (\ref{eq: mainbound}). Thus for $n\ge n(K)$, we have
\[\mathbf{Var}(\mathcal{N}_n[\varphi])\lesssim \|\varphi\|_{C^{1/2+\epsilon}}^2,\]
which is enough to close the approximation argument leading to the CLT.

To obtain the result concerning functions in $\dot{H}^{1/2}\cap L^\infty$, we choose $0<\delta<\epsilon_0$, and use the asymptotics in \cite{szego}, Theorem 8.22.29, and find that the rescaled Hermite polynomials are uniformly bounded in the region 
\[|x|,|y|\le 2-\delta.\]
Note that $\varphi(x)=\varphi(y)=0$ in the four rectangles (the ``corners'')
\begin{align*}
[-2,-2+\epsilon_0]\times[-2,-2+\epsilon_0] && [2-\epsilon_0,2]\times[-2,-2+\epsilon_0]\\
[2-\epsilon_0,2]\times[2-\epsilon_0,2] && [-2,-2+\epsilon_0,2-\epsilon_0,2].
\end{align*} 
The expression for the variance \eqref{eq: variancebound} is thus controlled by
\begin{multline*}C_{\delta}\int_{-2+\delta}^{2-\delta}\int_{-2+\delta}^{2-\delta}\frac{(\varphi(x)-\varphi(y))^2}{(x-y)^2}\,\mathrm{d}x\mathrm{d}y\\
+\iint_{R_1\cup R_2\cup R_3\cup R_4}\frac{(\varphi(x)-\varphi(y))^2}{(\delta-\epsilon_0)^2}(\psi_n(x)\psi_{n-1}(y)-\psi_n(y)\psi_{n-1}(x))^2\mathrm{d}x\mathrm{d}y,
\end{multline*}
where
\begin{align*}
R_1&=[-2,-2+\delta]\times[-2+\epsilon_0,2-\epsilon_0 & R_2&=[-2+\epsilon_0,2-\epsilon_0]\times[-2,-2+\delta]\\
R_3&=[2-\delta,2]\times[-2+\epsilon_0,2-\epsilon_0] & R_4 &= [-2+\epsilon_0,2-\epsilon_0]\times[2-\delta,2].
\end{align*}
Taking $\delta=\epsilon/2$, it now follows at once from the orthogonality of $\psi_n$ and the assumptions on $\varphi$ that
\[\mathbf{Var}(\mathcal{N}[\varphi])\le C_{\epsilon_0}\|\varphi\|^2_{\dot{H}^{1/2}}+C_{\epsilon_0} \|\varphi\|^2_{L^\infty}.\]
The CLT follows from this variance bound as in the previous cases.

\section{CLT for $C^{1/2+}$ and $H^{1/2+}$: A class of Wigner matrices}\label{sec: wigner}
In this section, we prove the central limit theorem for $C^{1/2+\epsilon}$ test functions for random matrices of Gaussian convolution type.  To achieve this, we will use a comparison procedure for the resolvent as in \cite{TV}. Let us first precisely describe the matrices we will be working with. The matrices we will consider have the form
\begin{equation} \label{eqn: jmatrixdef}
\sqrt{n}M= \frac{1}{2}(W+V),
\end{equation}
where $V$ is an independent GUE matrix and the distribution of the entries of $W=(w_{ij})_{1\le i, j \le n}$ satisfies condition \textbf{C0} (see (\ref{eqn: Czero})). The entries $\Re w_{ij}$, $\Im w_{ij}$, $1\le i < j \le n$ are all independent and the first five moments of the complex entries $w_{ij}$ match those of the GUE. That is 
\[\mathbf{E}[(\Re w_{ij})^\alpha (\Im w_{ij})^\beta]=\mathbf{E}X^\alpha Y^\beta\]
for any choice of  $0\le \alpha, \beta \le 5$ such that $\alpha+\beta=5$, where $X$ and $Y$ are independent $N(0,\frac{1}{2})$ random variables. We remark that the entries of $\sqrt{n}M$ have sub-exponential tails.

We denote by $\varphi_k$ the Littlewood-Paley projection of $\varphi$ onto frequencies of order $2^k$ (see Section \ref{sec: lp}). Write:
\begin{align*}
\varphi &= \left(\sum_{k\le \log n}+ \sum_{ \log n < k\le 2\log n } +\sum_{k>2\log n}\right)\varphi_k \\
&:= \varphi^{1,n}+\varphi^{2,n}+\varphi^{3,n}.
\end{align*}
If $\varphi \in C^{1/2+\epsilon}$, we use the trivial bound
\[\mathbf{E}|\mathcal{N}_n^\circ[\varphi]|^2\le 4n^2\|\varphi\|^2_{L^\infty}\]
and Lemma \ref{lem: hfcutoff}, and obtain
\[\mathbf{E}|\mathcal{N}_n^\circ[\varphi^{3,n}]|^2\le C_\epsilon n^{-\epsilon}.\]
Thus, the sequence $\mathcal{N}_n^\circ[\varphi^{3,n}]$ is a remainder whose second moment tends to zero. The contribution to $\mathcal{N}_n^\circ[\varphi]$ from $\varphi^{2,n}$ cannot be estimated quite so crudely. Nevertheless, we will show below that
\begin{equation}\label{eqn: secondremainder} \mathbf{E}|\mathcal{N}^\circ[\varphi^{2,n}]|^2\lesssim n^{-c\epsilon}
\end{equation}
for some $c>0$. Thus $\varphi^{2,n}$ does not contribute to the asymptotic distribution. 

Assuming (\ref{eqn: secondremainder}) for the moment, it suffices to show that $\mathcal{N}_n^\circ[\varphi^{1,n}]$ is asymptotically normal, with the correct limiting variance. For this, we will compare $\mathcal{N}^\circ[\varphi^{1,n}]$ to the corresponding quantity defined in terms of GUE eigenvalues. Let
\[\mathcal{N}_{n,\text{GUE}}^\circ[\varphi^{1,n}]=\sum_{j=1}^n\varphi^{1,n}(\lambda^{\text{GUE}}_j)-\mathbf{E}\sum_{j=1}^n\varphi^{1,n}(\lambda^{\text{GUE}}_j),\]
where $\lambda_j$, $1\le j\le n$ denote the eigenvalues of a GUE matrix normalized by $1/\sqrt{2n}$. We express the variance in terms of the GUE quantity:
\[\mathbf{Var}(\mathcal{N}_n^\circ[\varphi^{1,n}]) = \mathbf{Var}(\mathcal{N}_{n,\text{GUE}}^\circ[\varphi^{1,n}]) + \left(\mathbf{Var}(\mathcal{N}_{n,\text{GUE}}^\circ[\varphi^{1,n}])-\mathbf{Var}(\mathcal{N}_{n}^\circ[\varphi^{1,n}])\right).\]
By (\ref{eqn: guevariancebound}), the first term is bounded. Using equation (\ref{eq: equality}), we estimate the second term by
\begin{equation}\label{eqn: covcomparison}
n^2\sum_{-1 \le k,l \le \log n }\iint|C(z^k_1,z^l_2;M)-C(z^k_1,z^l_2;\text{GUE})||g_k(s)||g_l(t)|\,\mathrm{d}t\mathrm{d}s.
\end{equation}
Here, $z^k_1 = s+i\cdot2^{-k}$, $z^l_2 = t+i\cdot 2^{-l}$, and
\begin{align*}
C(z_1,z_2;M)&=\mathbf{Cov}(\Im s(z_1),\Im s(z_2)),\\
C(z_1,z_2;\text{GUE})&=\mathbf{Cov}(\Im s_{\text{GUE}}(z_1),\Im s_{\text{GUE}}(z_2)),\\
s(z) &=\frac{1}{n}\operatorname{tr}(M-z)^{-1},\\
s_{\text{GUE}}(z) &= \frac{1}{n}\operatorname{tr}\left(\frac{1}{\sqrt{2n}}V-z\right)^{-1}.
\end{align*}


We split the integral in \eqref{eqn: covcomparison} into two regions:
\begin{multline}\label{eq: splitintegral2}
n^2\iint_{[-3,3]\times [-3,3]}|C(z_1^k,z_2^l;M)-C(z_1^k,z_2^l;\text{GUE})||g_k(s)||g_l(t)|\,\mathrm{d}t\mathrm{d}s\\
+n^2\iint_{([-3,3]\times [-3,3])^c}|C(z_1^k,z_2^l;M)-C(z_1^k,z_2^l;\text{GUE})||g_k(s)||g_l(t)|\,\mathrm{d}t\mathrm{d}s.
\end{multline}
We first show that the integrals
\begin{equation} \label{eqn:term1}
n^2\iint_{[-3,3]\times [-3,3]}|C(z_1^k,z_2^l;M)-C(z_1^k,z_2^l;\text{GUE})||g_k(s)||g_l(t)|\,\mathrm{d}t\mathrm{d}s
\end{equation}
are bounded by $\|\varphi\|_{C^{1/2+\epsilon}}$, with factors summable in $k$ and $l$.

To compare $C(z_1^k,z_2^l;M)$ and $C(z_1^k,z_2^l;\text{GUE})$, we shall make use of the following slight modification of a result of T. Tao and V. Vu \cite{TV}:

\begin{proposition}[Tao, Vu] \label{prop:comp}
Let $M_0$ be a random Hermitian matrix where each entry satisfies condition (\textbf{C0}), and let $\xi$ be a random variable independent of $M_0$.  Let
\[M = M_0 + \frac{1}{\sqrt{n}}\xi A,\]
where $A$ is an elementary matrix.  Denote the Stieltjes transform of the empirical eigenvalue densities of $M_0$ and $M$ by $s_0$ and $s$, respectively.  Then, for fixed $m \geq 1$ and uniformly over $|E| < 3, \eta > n^{-1-\delta}$ for some sufficiently small $\delta > 0$ and with overwhelming probability, we have
\begin{equation} \label{eqn:comparison}
s(E+i\eta) = s_0(E+i\eta) + \sum_{j=1}^m\xi^jn^{-j/2}c_j(\eta) + O(n^{-(m+1)/2(1 - C\delta)})\min\{1,\frac{1}{n\eta}\},
\end{equation}
where 
\[c_j(\eta) = O(n^{Cj\delta} \min\{1,\frac{1}{n\eta}\}).\]

\end{proposition}
The proof of the above proposition follows as in \cite{TV}.  To derive the main result in \cite{TV}, Tao and Vu only needed to compare resolvents for a fixed energy $E$.  Using Theorem \ref{thm:EYY}, together with the bound
\[ \frac{\Im s(E+i\eta_1)}{\Im s(E+i\eta_2)} \leq \frac{\eta_2}{\eta_1},\]
Lemma 16 in \cite{TV} will hold uniformly over the range $|E|<3$. By \cite{TV}, Proposition 13, this yields the expansion above. In particular, the condition that $\eta > n^{-1-\delta}$ is crucial. 

Setting
\[S(m,\eta) = \sum_{j=1}^m\xi^jn^{-j/2}c_j(\eta) \]
and taking imaginary parts on both sides of equation (\ref{eqn:comparison}), we have
\begin{multline} \label{eqn: product}
\mathbf{E}(\Im s(z^k_1) \Im s(z^k_2)) = \mathbf{E}(\Im s_0(z^k_1)\Im s_0(z^l_2))  \\
+\mathbf{E}\left[\Im s_0(z^k_1)\cdot \Im \left( S(m,2^{-l})+O(n^{-\frac{m+1}{2}(1-C\delta)})\cdot \min(1,n^{-1}2^{l})\right)\right] \\
+\mathbf{E}\left[\Im s_0(z^k_2)\cdot \Im \left( S(m,2^{-k})+O(n^{-\frac{m+1}{2}(1-C\delta)})\cdot \min(1,n^{-1}2^{k})\right)\right] \\
+\mathbf{E}\Im \left( S(m,2^{-l})+O(n^{-\frac{m+1}{2}(1-C\delta)})\cdot \min(1,n^{-1}2^{l})\right)  \\
\times \Im \left( S(m,2^{-k})+O(n^{-\frac{m+1}{2}(1-C\delta)})\cdot \min(1,n^{-1}2^{k})\right). 
\end{multline}

In the above, we have omitted error terms from exceptional sets whose complements have overwhelming probability. This is justified since all quantities are polynomially bounded in $n$ for $k,l\le C\log n$. We will apply the above estimate in the situation where variables $\xi_1, \xi_2$ take the place of $\xi$ and the moments of $\xi_1,\xi_2$ match up to fifth order. When we take the difference, the terms involving moments of fifth order or less can be ignored.  However, we cannot afford the error terms appearing on the second and third lines of (\ref{eqn: product}). On the other hand, as will be seen below, the factor $\min(1,n^{-1}2^{k})\cdot \min(1,n^{-1}2^{l})$ appearing in the final term on the right side of (\ref{eqn: product}) is sufficient to allow control of the difference of covariances. 

The concentration of $\Im s_0$ around its mean provided by \ref{thm:EYY} allows us to rewrite the second and third lines of \eqref{eqn: product} as
\begin{equation}\label{eqn: firstrep}
\mathbf{E}\Im s_0(z^k_1)(1+O(2^k/n^{1-\delta})) \cdot \mathbf{E} \Im \left( S(m,2^{-l})+O(n^{-\frac{m+1}{2}(1-C\delta)})\cdot \min(1,n^{-1}2^{l})\right),
\end{equation}
and
\begin{equation}\label{eqn: secondrep}
\mathbf{E}\Im s_0(z^k_2)(1+O(2^l/n^{1-\delta}))\cdot \mathbf{E}\Im \left( S(m,2^{-k})+O(n^{-\frac{m+1}{2}(1-C\delta)})\cdot \min(1,n^{-1}2^{k})\right).
\end{equation}

We now use (\ref{eqn:comparison}) in the second term $\mathbf{E}(s(z_1))\mathbf{E}(s(z_2))$ to obtain the cancellation we need:
\begin{multline} \label{eqn: prodofe}
\mathbf{E}(\Im s(z^k_1))\mathbf{E}(\Im s(z^l_2)) = \mathbf{E}(\Im s_0(z^k_1)) \mathbf{E}(\Im s_0(z^l_2)) \\ 
+ \mathbf{E}\Im s_0(z^k_1) \mathbf{E}\Im \left(S(m,2^{-l}) + O(n^{-\frac{m+1}{2}(1 - C\delta)})\min(1,n^{-1}2^{l})\right)\\
+ \mathbf{E}\Im s_0(z^l_2) \mathbf{E}\Im \left(S(m,2^{-k}) + O(n^{-\frac{m+1}{2}+(1 - C\delta)})\min(1,n^{-1}2^{k})\right) \\ 
+ \mathbf{E}\Im \left(S(m,2^{-k}) + O(n^{-\frac{m+1}{2}+(1 - C\delta)})\min(1,n^{-1}2^{k})\right) \\ 
\times \mathbf{E}\Im \left(S(m,2^{-l}) + O(n^{-\frac{m+1}{2}(1 - C\delta)})\min(1,n^{-1}2^{l})\right).
\end{multline}
The terms on the second and third lines of (\ref{eqn: prodofe}) cancel the terms \ref{eqn: firstrep} and \ref{eqn: secondrep}, respectively, up to an error of order
\[O(n^{-\frac{m+1}{2}(1 - (C+1)\delta)})\min(1,2^k/n)\min(1,2^l/n).\]  
If we take $m\ge 5$ this error is
\[n^{-3+\epsilon'}\min(1,2^k/n)\min(1,2^l/n),\]
for some small $\epsilon'$ depending on $\delta$. 
 We have now established that for two matrices $M$, $M^\prime$ differing in one pair of entries whose distributions have $5$ matching moments and for some small $\epsilon' > 0$,
\begin{equation} \label{eqn:oneentry}
|C(z^k_1,z^l_2;M)-C(z^k_1,z^l_2;M')| \leq O(n^{-3+\epsilon'})\min(1,n^{-1}2^{k})\min(1,n^{-1}2^{l}).
\end{equation}

We have to swap $O(n^2)$ entries to bridge between the GUE and our Wigner matrix.  Applying equation (\ref{eqn:oneentry}) repeatedly and taking into account the factor of $n^2$ in equation (\ref{eqn:term1}), we see that if two matrices match up to $5$ moments, then the difference in the covariance is bounded by 
\begin{equation} \label{eqn:covdiffbd}
n^{1+\epsilon'}\cdot \min\{1,\frac{1}{n2^{-k}}\}\min\{1,\frac{1}{n2^{-l}}\}.
\end{equation}

The estimate (\ref{eqn:covdiffbd}) allow us to close the argument. To see this, we sum over $k,l$ of (\ref{eqn:term1}) and using (\ref{eqn:covdiffbd}), we have
\begin{align} \nonumber
n^2\sum_{-1 \le k,l \le \log n}\iint_{[-3,3]\times [-3,3]}&|C(z_1^k,z_2^l;M)-C(z_1^k,z_2^l;\text{GUE})||g_k(s)||g_l(t)|\,\mathrm{d}t\mathrm{d}s \\ \nonumber
&\lesssim n^{1+\epsilon'}\left(\sum_{-1\le k \le\log n}\|g_k\|_{L^\infty}\min(1,n^{-1}2^{k})\right)^2.
\end{align}

Inside the bracket, we have
\begin{align*} 
\sum_{-1\le k \le \log n}\|g_k\|_{L^\infty}\min(1,2^{k}n^{-1}) &= \sum_{-1\le k\le \log n}\|g_k\|_{L^\infty}2^{k(\frac{1}{2}+\epsilon)}2^{k(\frac{1}{2}-\epsilon)}\cdot n^{-1}\\
&\lesssim n^{-1}\sup_j (2^{j(\frac{1}{2}+\epsilon)}\|g_j\|_{L^\infty}) \\
&\quad \times \sum_{-1\le j\le \log n}2^{k(\frac{1}{2}-\epsilon)}\\
&\lesssim n^{-\frac{1}{2}-\epsilon'}\|\varphi\|_{C^{1/2+\epsilon}}.
\end{align*}
Squaring both sides, we gain a factor of $n^{-1-2\epsilon}$ which allows us to cancel out the factor of $n^{1+\epsilon'}$ provided we pick $\epsilon'$ sufficiently small in (\ref{eqn:oneentry}), thus ensuring that the error stays bounded up to a constant (in fact $o(1)$) factor by $\|\varphi\|_{C^{1/2+\epsilon}}^2$.  In conclusion, we have bounded $\mathbf{Var}(\mathcal{N}_n^\circ[\varphi^{1,n}])$ by $C\|\varphi\|_{C^{1/2+\epsilon}}^2$.

We now deal with the integrals in \eqref{eq: splitintegral} over the region $([-3,3]\times [-3,3])^c$. We illustrate the procedure with the integral over the infinite rectangle 
\[\{(s,t): (s,t)\in [-3,3]\times [3,\infty)\},\] the other cases being similar. We use the deterministic version of \eqref{eqn:comparison} in \cite{TV}, as we have done in the case $s,t\in [-3,3]$ above. However, instead of bounding the coefficients $c_j(\eta)$ as in Proposition \ref{prop:comp}, we return to the expansion in \cite{TV}, Proposition 13. We expand both $\Im s(z^k_1)$, $\Im s(z^l_2)$ to 8th order as above, and perform all the same cancellations, but use a different bound for the coefficients appearing in the expansion of $ s(z^l_2)$, the resolvent at energy $t\in [3,\infty)$. The estimate provided by \cite{TV} for $c_j(\eta)$ is:
\[c_j=O(\|R_0\|^j_{(\infty,1)}\min(\|R_0\|_{(\infty,1)},1/n\eta)),\]
where $R_0$ is the resolvent of the unperturbed matrix $M_0$ in the statement of Proposition \ref{prop:comp}, and
\[\|R_0\|_{(\infty,1)}= \max_{1\le i,j\le n}|(R_0)_{ij}|.\]
On the event $\max \{|\lambda_1|,|\lambda_n|\}\le t/2$, we have 
\[\max_i \frac{1}{|\lambda_i-t|}\lesssim \frac{1}{t}\]
for $t\in [3,\infty)$. By \cite{EYY3}, Theorem 7.2, the complement of this event has probability
\[Ce^{-n^c\log t} \lesssim 3^{-n^c/2}t^{-n^c/2}.\]
The conjunction of the event and delocalization of eigenvectors, a consequence of the local-semicircle law, implies, as in \cite{TV}, Lemma 16, that
\[c_j(\eta) = O(n^{\epsilon''}t^{-j})\min(n^{\epsilon''}t^{-1},1/n\eta).\]
uniformly in $\eta>n^{-1-\delta}$. Using these observations, we obtain the bound:
\[|C(z^k_1,z^l_2;M)-C(z^k_1,z^l_2;M')| \leq O(n^{-3+\epsilon'})t^{-2}\min(1,n^{-1}2^{k})\min(1,n^{-1}2^{l}),\]
for $s,t \in [-3,3]\times[0,3)$, which allows us to perform the integral over the infinite range and conclude the argument as previously.

It only remains to justify equation (\ref{eqn: secondremainder}). Only here do we utilize the Gaussian convolution structure. We begin by stating the following results due to E. Br\'ezin  and S. Hikami \cite{brezinhikami}, and K. Johansson \cite{johansson2}, (Proposition 1.1, Lemma 2.1, Lemma 3.2).
\begin{proposition} [Br\'ezin-Hikami, Johansson]
The symmetrized eigenvalue measure on $\mathbb{R}^n$ induced by the Johansson matrix $M$ defined in (\ref{eqn: jmatrixdef}) has a density $\rho_n(x)$, given by
\begin{equation}
\rho_n(x) = \int_{\mathcal{H}_N}q_{a^2/n}(x;z(H))\,\mathrm{d}P^{(n)}(H),
\end{equation}
where $a^2=\frac{1}{4}$ and
\[q_{a^2/n}(x;z)=\left(\frac{n}{2\pi a^2}\right)^n\frac{\Delta_n(x)}{\Delta_n(y)}\mathrm{det}(e^{-n(x_j-z_k)^2/(2a^2)})^n_{j,k=1}.\]
$\Delta_n =\prod_{i<j}|x_i-x_j|$ is the Vandermonde determinant.
\end{proposition}
$\mathrm{d}P^{(N)}(H)$ denotes the measure on the space of Hermitian matrices induced by the entry distribution of the matrix $H=\frac{1}{2\sqrt{n}}W$ (see  (\ref{eqn: jmatrixdef})), and $z=\{z_j(H)\}_{j=1}^n$ is the vector of eigenvalues of $H$.
\begin{proposition}[Br\'ezin-Hikami, Johansson]
The correlation functions for $q_{a^2/n}(x;z)$ are given by
\begin{align}
R_m^n(x_1,\hdots,x_m;z) &= \frac{n!}{(n-m)!}\int_{\mathbb{R}^{N-m}}q_{a^2/n}(x;z)\mathrm{d}x_{m+1}\hdots \mathrm{d}x_n \\
&= \emph{det}(K_n^{a^2/n}(x_i,x_j;z))_{i,j=1}^m,
\end{align}
for some kernel $K_n$.
\end{proposition}
In equation (2.19) in \cite{johansson2}, Johansson gives an alternative form for the determinantal kernel appearing in the previous proposition:
\begin{equation*}
\mathcal{K}_n(u,v;z) = e^{\frac{n(u^2-v^2)}{2a^2}+\omega(u-v)}K_n^{a^2/n}(u,v;z),
\end{equation*}
where $\omega$ is any constant, $\mathcal{K}_n(u,v;z)$ also serves as a determinantal kernel which produces the same correlation functions $q_{a^2/n}(x;z)$.

We use the following variant of the asymptotic analysis of the diagonal kernel $\mathcal{K}_n(u,u;z)$ in \cite{johansson}, derived using ideas from \cite{erdosschleinpeche}. The proof will be found in the next section.
\begin{lemma}\label{lem:diagonalestimate}
Let $\delta_1,\delta_2>0$. Let $\Omega_{R,n}=\{s\in \mathbb{C}: |\Re s|\le R,  n^{-1/6}\le |\Im s| \le R\}$. Let $Z_n$ be the set of $z\in \mathbb{R}^n$ for which the local semicircle law holds, that is, the set of $z$ such that the event on the left side of \eqref{eq: rigidity} is satisfied with $m(z)=\frac{1}{n}\sum_j \frac{1}{z-z_j}$. Then, for $z\in Z_n$ and any $x$ such that $|x|> \sqrt{2}- n^{-1/3+\delta_2}$, we have the estimate
\[\left|\frac{1}{n\rho_{sc}(x)}\mathcal{K}_n(x,x;z)-1\right|\lesssim_{\delta_{1,2}} n^{-1/2+\delta_1}.\]
\end{lemma}
Here, $\rho_{sc}$ denotes the semicircle density:
\[\rho_{sc}(x) =\frac{1}{\pi}\sqrt{2-x^2}\mathbf{1}_{[-\sqrt{2},\sqrt{2}]}(x).\]

Assuming the lemma for the moment, we proceed to bound the variance of $\mathcal{N}_n[\varphi^{2,n}]$. In what follows, it will be useful to avoid the spectral edges. For this purpose, we let $\theta_n(x)$ be a smooth, real function such that  $0\le \theta_n \le 1$, 
\[\theta_n(x)=1, \quad x\in [-\sqrt{2}+n^{-1/3+\delta_2}/2,\sqrt{2}-n^{-1/3+\delta_2}/2],\] 
and 
\[\theta_n(x)=0, \quad \text{ for } x\in [-\sqrt{2}+n^{-1/3+\delta_2},\sqrt{2}-n^{-1/3+\delta_2}]^c.\]
Then
\[\mathbf{Var}(\mathcal{N}_n[\varphi^{2,n}])\le 2\mathbf{Var}(\mathcal{N}_n[\theta_n\varphi^{2,n}])+ 2\mathbf{Var}(\mathcal{N}_n[(1-\theta_n))\varphi^{2,n}]).\]
The rigidity of eigenvalues theorem \cite{EYY} (Theorem \ref{thm:EYY} above), implies that 
\[|\{\lambda_j: |\lambda_j|>\sqrt{2}-n^{-1/3+\delta_2}\}| \lesssim n^{1/2+2\delta_2}\]
with overwhelming probability. Since $\|(1-\theta_n) \varphi^{2,n}\|_{L^\infty}\le \|\varphi^{2,n}\|_{L^\infty}$, we have:
\[\mathbf{Var}(\mathcal{N}_n[(1-\theta_n))\varphi^{2,n}]) \lesssim  \|\varphi^{2,n}\|_{L^\infty}^2n^{1+4\delta_2}.\]
By the assumption $\varphi\in C^{1/2+\epsilon}$ and Lemma \ref{lem: hfcutoff}, we have,
\[\|\varphi^{2,n}\|^2_{L^\infty} \le \|\varphi\|^2_{C^{1/2+\epsilon}}n^{-1-2\epsilon}.\]
It follows that 
\[\mathbf{Var}(\mathcal{N}_n[\varphi^{2,n}])\lesssim \mathbf{Var}(\mathcal{N}_n[\theta_n\varphi^{2,n}]),\]
provided $2\epsilon > 4\delta_2$.
Together with the $L^\infty$ bound
\[\|\theta\varphi^{2,n}\|_{L^\infty}\le \|\varphi^{2,n}\|_{L^\infty},\]
this justifies the simplifying assumption, which we make from now on, that $\varphi^{2,n}(x)= 0$ for $|x|>\sqrt{2}-n^{-1/3+\delta_2}$.

 Using equations (\ref{eqn: correlationkernel}), (\ref{eqn: squarekernel}), and (\ref{eqn: doubletrace}) following the derivation of (\ref{eq: variancebound}), we write 
\begin{equation}\label{eqn: T2variancebound}
\mathbf{Var}(\mathcal{N}_n[\varphi^{2,n}])=\iint |\varphi^{2,n}(x)-\varphi^{2,n}(y)|^2\,T_2(x,y)\,\mathrm{d}x\mathrm{d}y,
\end{equation}
where 
\[T_2=n^2\rho_{n,1}(x)\rho_{n,1}(y)-n(n-1)\rho_{n,2}(x,y),\] 
and $\rho_{n,1}$, $\rho_{n,2}$ denote the one- and two-point correlation functions of $M$. 
For technical reasons which will become clear below, we remove a small neighborhood of the diagonal $x=y$ in \eqref{eqn: T2variancebound}. Define 
\[\Delta=\{(x,y):|x-y|\le n^{-1+\delta}\}.\]
We write
\begin{align*}
\iint |\varphi^{2,n}(x)-\varphi^{2,n}(y)|^2\,T_2(x,y)\,\mathrm{d}x\mathrm{d}y &= \iint_{\Delta^c} |\varphi^{2,n}(x)-\varphi^{2,n}(y)|^2\,T_2(x,y)\,\mathrm{d}x\mathrm{d}y\\
&\quad + \iint_{\Delta} |\varphi^{2,n}(x)-\varphi^{2,n}(y)|^2\,T_2(x,y)\,\mathrm{d}x\mathrm{d}y.
\end{align*}
The ``diagonal'' term in the previous equation has the following simple estimate:
\begin{multline}\label{eqn: diagonalbound}
\left| \iint_\Delta |\varphi^{2,n}(x)-\varphi^{2,n}(y)|^2\,T_2(x,y)\,\mathrm{d}x\mathrm{d}y\right| \\ \lesssim n^2\iint_{|x-y|<n^{-1+\delta}} f(x,y)(\rho_{n,1}(x)\rho_{n,2}(y)+\rho_{2,n}(x,y))\,\mathrm{d}x\mathrm{d}{y},
\end{multline}
where $|f(x,y)|\lesssim \|\varphi^{2,n}\|_{L^\infty}^2$. We cover the diagonal of $[-3,3]\times[-3,3]$ by $n^{1-\delta_3}$ squares $S_j=I_j\times I_j$ with $|I_j|\lesssim n^{-1+\delta_3}$, and obtain the bound
\begin{align*}\|\varphi^{2,n}\|_{L^\infty}^2n^2\sum_j \iint_{I_j\times I_j}\rho^{n,1}(x)\rho^{n,1}(y)\,\mathrm{d}x\mathrm{d}y&\lesssim n^{-1-\epsilon}\sum_j \left(\mathbf{E}|\lambda_i:\lambda_i \in I_j|\right)^2\\
&\lesssim n^{-1-\epsilon}n^{1-\delta_3} n^{2\delta_3}\\
&\le n^{\delta_3-\epsilon}.
\end{align*}
The term involving $\rho^{2,n}(x,y)$ in \eqref{eqn: diagonalbound} is bounded similarly.

Using the determinantal structure, we can express $T_2$ in terms of $\mathcal{K}_n$ as
\begin{align} \label{eqn:T2}
T_2(x,y) &= \int \mathcal{K}_n(x,x;z) dP^{(N)}(z(H)) \int\mathcal{K}_n(y,y;z)\,\mathrm{d}P^{(n)}(z(H))\\ \nonumber
& \quad - \int \left(\mathcal{K}_n(x,x;z)\mathcal{K}_n(y,y;z) -  \mathcal{K}_n(x,y;z)\mathcal{K}_n(y,x;z)\right) \,\mathrm{d}P^{(n)}(z(H))\\
& = \int \mathcal{K}_n(x,x;z) \,\mathrm{d}P^{(n)}(z(H)) \int\mathcal{K}_n(y,y;z) \,\mathrm{d}P^{(n)}(z(H)) \nonumber\\ 
& \quad - \int \mathcal{K}_n(x,x;z)\mathcal{K}_n(y,y;z)\,\mathrm{d}P^{(n)}(z(H)) +\int \mathcal{K}_n(x,y;z)\mathcal{K}_n(y,x;z) \,\mathrm{d}P^{(n)}(z(H)).\nonumber
\end{align} 
We insert the final expression for $T_2$ in \eqref{eqn:T2} into (\ref{eqn: T2variancebound}) and find
\begin{align}\label{eqn: expandedvariance}
\mathbf{Var}(\mathcal{N}_n[\varphi^{2,n}])&=\iint |\varphi^{2,n}(x)-\varphi^{2,n}(y)|^2 \Big( \int \mathcal{K}_n(x,x;z) \,\mathrm{d}P^{(n)}(z(H)) \\
& \times \int\mathcal{K}_n(y,y;z) \,\mathrm{d}P^{(n)}(z(H)) \,\mathrm{d}x\mathrm{d}y - \int \mathcal{K}_n(x,x;z)\mathcal{K}_n(y,y;z)\, \mathrm{d}P^{(n)}(z(H))\Big)\,\mathrm{d}x\mathrm{d}y \nonumber \\
&\quad +\iint|\varphi^{2,n}(x)-\varphi^{2,n}(y)|^2 \int \mathcal{K}_n(x,y;z)\mathcal{K}_n(y,x;z)\,\mathrm{d}P^{(n)}(z(H))\,\mathrm{d}x\mathrm{d}y. \nonumber
\end{align}
First, consider the contribution from the final term
\begin{equation}\label{eqn: xyterm}
\iint|\varphi^{2,n}(x)-\varphi^{2,n}(y)|^2 \int \mathcal{K}_n(x,y;z)\mathcal{K}_n(y,x;z)\,\mathrm{d}P^{(n)}(z(H))\,\mathrm{d}x\mathrm{d}y.\end{equation}
The saddle point analysis in the next section yields the following bound for the product of the kernel evaluated at two off-diagonal points on a set of $\mathrm{d}P(z)$-overwhelming probability.
\begin{equation}\label{eqn: offdiagbound}
|\mathcal{K}_n(x,y;z)\mathcal{K}_n(y,x;z)|\lesssim |x-y|^{-2}
\end{equation}
in the region
\[|x-y|>n^{-1+\delta}, \quad |x|, |y|<\sqrt{2}-n^{-1/3+\delta_2}.\]
Thus, the second term in (\ref{eqn:T2}) is bounded up to a constant by
\[\|\varphi\|^2_{C^{1/2+\epsilon}}n^{-1-2\epsilon} \iint_{|x-y|>n^{-1+\delta}} \frac{1}{|x-y|^2}\,\mathrm{d}x\mathrm{d}y \lesssim \|\varphi\|^2_{C^{1/2+\epsilon}}n^{-c},\]
for $c>0$. We have bounded the contribution to \eqref{eqn:T2} from \eqref{eqn: expandedvariance}. The error terms resulting from sets $A$ of negligible probability are easily controlled by the estimate
\begin{multline*}\|\varphi^2_n\|_{L^\infty}^2 \int _A \iint |\mathcal{K}_n(x,y;z)\mathcal{K}_n(y,x;z)|\,\mathrm{d}x\mathrm{d}y \,\mathrm{d}P^{(n)}(z(H))\\ \lesssim \|\varphi^2_n\|^2_{L^\infty}\int_A \iint n^2(\rho_{n,1}(x)\rho_{n,1}(y)+\rho_{2,n}(x,y))\mathrm{d}x\mathrm{d}y \,\mathrm{d}P^{(n)}(z(H)).\end{multline*}

Using Lemma \ref{lem:diagonalestimate}, we can write
\begin{equation}
\label{eqn: densitybound}
\mathcal{K}_n(x,x;z) = n(\rho_{sc}(x) + \psi(x;z)),
\end{equation}
with $|\psi(x;z)| \le C_{\delta_1,\delta_2} n^{-1/2+\delta_1}$ for $|x| < \sqrt{2}-n^{-1/3+\delta_2}$ and $z$ in a set of overwhelming probability. When $|x| <\sqrt{2}- n^{-1/3+\delta_2}$, or $|y| <\sqrt{2}- n^{-1/3+\delta_2}$, we use (\ref{eqn: densitybound}), together with the local semicircle law to show that the double integral over $x$ and $y$ of the difference
\[\int \mathcal{K}_n(x,x;z)\mathcal{K}_n(y,y;z) \mathrm{d}P^{(n)}(z(H)) - \int \mathcal{K}_n(x,x;z) \mathrm{d}P^{(n)}(z(H)) \int\mathcal{K}_n(y,y;z)\mathrm{d}P^{(n)}(z(H))\]
is bounded by $n^2n^{-1+2\delta_2}$. In deriving this bound, we must exclude some sets whose complements have $\mathrm{d}P^{(n)}(z(H))$-overwhelming probability. The corresponding error terms are controlled by noting that
\[\int \mathcal{K}_n(x,x;z)\,\mathrm{d}x=n,\]
whatever the value of $z$.

By Fubini's theorem, we obtain a bound of $C\|\varphi^{2,n}\|^2_{L^\infty}n^2n^{-1+2\max(\delta_2,\delta_1)}$ for first term on the right side in (\ref{eqn: expandedvariance}). Choosing $\delta_1, \delta_2, \delta_3 < \epsilon/20$, we have established (\ref{eqn: secondremainder}).

To prove the result for $H^{1/2+\epsilon}$ test functions $\varphi$ that are supported on $(-\sqrt 2+\epsilon_0, \sqrt 2-\epsilon_0)$, we follow a similar strategy as above.  Before splitting $\varphi$ into Littlewood-Paley components and by our support assumption, we may insert a smooth cutoff $0\le \chi \le 1$ multiplying each $g_k$, where $\chi$ is $1$ on $(-\sqrt 2+\epsilon_0, \sqrt 2-\epsilon_0)$ and $0$ outside $(-\sqrt 2+\epsilon_0/2, \sqrt 2-\epsilon_0/2)$. It is clear that $\|\chi g_k\|_{L^p}\le \|g_k\|_{L^p}$.

For the Littlewood-Paley components $k \leq \log n$, we perform the comparison analysis as above, noting that the analogous equation to equation (\ref{eqn:term1}) contains only the first part, with the interval $[-3,3]\times [-3,3]$ replaced by $[-\sqrt{2}+\epsilon_0/2, \sqrt{2}-\epsilon_0/2] \times [-\sqrt{2}+\epsilon_0/2, \sqrt{2}-\epsilon_0/2]$.  The rest follows from the simple estimate $\|\chi g_k\|_{L^2} \lesssim \|\chi g_k\|_{L^\infty}$ and the characterisation of $H^{1/2+\epsilon}$ spaces in terms of Littlewood-Paley decomposition.

Let $p > \frac{\log n}{\epsilon}$, so that $2^{-p\epsilon} \le \frac{1}{n}$.  Then the sum over the components $k \geq p$ is bounded trivially by:
\begin{align*}
\mathbf{Var}(\mathcal{N}_n[\sum_{k>p}\varphi_k]) &\leq n^2 \|\sum_{k>p}\varphi_k\|^2_{L^{\infty}} \\
&\lesssim 2^{-2p\epsilon}n^2\|\sum_{k>p}\varphi_k\|^2_{H^{1/2+\epsilon}} \\
&\lesssim \|\sum_{k>p}\varphi_k\|^2_{H^{1/2+\epsilon}}.
\end{align*}

Finally for frequencies where $\log n < k < \frac{1}{\epsilon}\log n$, denote by $\psi := \sum_{\log n < k \le \frac{1}{\epsilon} \log n}\chi \varphi_k$, we have from equations (\ref{eqn: xyterm}), (\ref{eqn: offdiagbound}), the Cauchy-Schwarz inequality and following a similar analysis as above, using in particular that $\chi$ is zero close to the spectral edges:
\begin{align*}
\mathbf{Var}(\mathcal{N}_n[\psi]) &\lesssim \iint_{|x-y|>n^{-1+\delta}} |\psi(x)-\psi(y)|^2 \frac{1}{|x-y|^2} \mathrm{d}x\mathrm{d}y.\\
& \lesssim n^{1-\delta}\|\psi\|^2_{L^2} \\
& \lesssim n^{-\delta-\epsilon}\|\psi\|^2_{H^{1/2+\epsilon}}.
\end{align*}

\section{Saddle point analysis of the determinantal kernel}\label{sec: saddlept}
In this section, we justify the asymptotic approximations \eqref{eqn: offdiagbound} and \eqref{eqn: densitybound}. For this, we use the representation of $K_n^{a^2/n}$ given by E. Br\'ezin and S. Hikami \cite{brezinhikami}, and further analyzed by K. Johansson \cite{johansson2}:
\begin{align}
K^{1/4n}_n(x,y;z)&=\frac{e^{2n(x^2-y^2)}n}{(\pi i)^2 }\int_\gamma\int_{\Gamma_L} \frac{1}{w-s}e^{n(f_{n,x}(w)-f_{n,y}(s))}\,\mathrm{d}w\mathrm{d}s \label{eq: crossingcontours}\\
f_{n,x}(w) &= 2(w^2-2xw)+\frac{1}{n}\sum_{j=1}^n\log(w-z_j).\nonumber
\end{align}
$\gamma$ is a rectangular contour 
\[\gamma=\bigcup_{i=1}^4 \gamma_i.\]
\begin{align*}
\gamma_1 :\ & t\rightarrow -M+it, \quad  -ib \le t\le ib; &\gamma_2 :\ & t\rightarrow  t+ib, \quad    -M\le t\le M.\\
\gamma_3 &= -\gamma_1 & \gamma_4 &=\overline{\gamma_2}.
\end{align*}
$M$ is chosen such that $\gamma$ surrounds all the $z_i$ $1\le i \le n$, and $\Gamma_L$ is the vertical contour $\Gamma_L: t\mapsto it+L$, where $0<M<L$ is chosen large enough to make $\Gamma_L$ and $\gamma$ disjoint (Figure \ref{fig:cont1}).
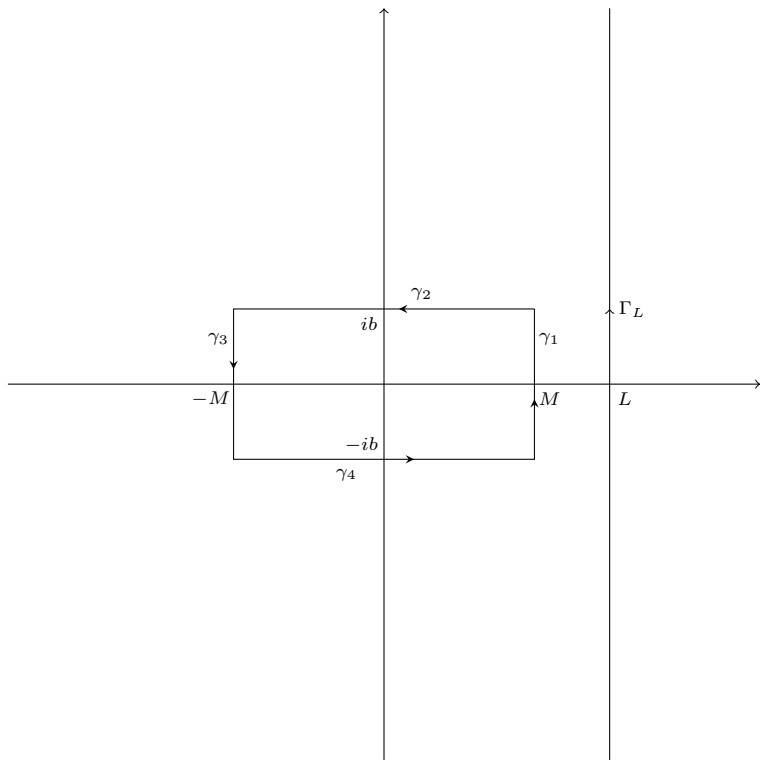
\begin{figure}[h]
\centering
\begin{tikzpicture}[decoration={markings,
     mark=at position 0.15 with {\arrow{stealth};},
     mark=at position 0.4 with {\arrow{stealth};},
     mark=at position 0.7 with {\arrow{stealth};},
     mark=at position 0.9 with {\arrow{stealth};}}]
\draw[->] (-5,0) -- (5,0); 
\draw[->] (0,-5) -- (0,5);
\draw [postaction={decorate}] (2,1) -- (-2,1) -- (-2,-1) -- (2,-1) -- (2,1);
\draw[->] (3,-5)--(3,1);
\draw (3,1)--(3,5);
\tiny \draw (2.2,-0.2) node{$M$};
\draw (-2.3,-0.2) node{$-M$};
\draw (2.2,0.6) node{$\gamma_1$};
\draw (0.5,1.2) node{$\gamma_2$};
\draw (-2.2,0.6) node{$\gamma_3$};
\draw (-0.5,-1.2) node{$\gamma_4$};
\draw (3.2,-0.2) node{$L$};
\draw (3.3,1) node{$\Gamma_L$};
\draw (-0.2,0.8) node{$ib$};
\draw (-0.3,-0.8) node{$-ib$};
\end{tikzpicture}
\caption{Contours $\gamma$ and $\Gamma$}
\label{fig:cont1}
\end{figure}
To prove Lemma \ref{lem:diagonalestimate}, we can follow \cite{erdosschleinpeche} fairly closely. The starting point is the following alternate representation of $K^{1/4n}_n$, due to Johansson, which removes the singularity of the integrand:
\begin{align}K^{1/4n}_n(x,y;z)&=\int_\gamma \int_\Gamma h(s,w)g_n(s,w) e^{n(f_n(w)-f_n(s))}\,\mathrm{d}w\mathrm{d}s, \label{eq: johanssonform} \\
g_n(s,w)&=\frac{4}{z}\left(w+z-y-\frac{1}{4}\sum_{j=1}^n\frac{z_j}{(w-z_j)(s-z_j)}\right), \nonumber \\
h(s,w)&=\frac{e^{4n(x-y)w}(e^{4n(x-y)w}-e^{4n(x-y)(w-s)})}{y-x}, \nonumber \\
f_n&=f_{n,x}(w).\nonumber
\end{align}
As noted in \cite{johansson2}, we have\footnote{Note that there is a typographical error in the corresponding equation in the print version of \cite{johansson2}.}
\[g_n(s,w) = \frac{1}{s}f'_n(w)+\frac{f_n'(s)-f_n'(w)}{s-w}.\]
The result of the saddle point analysis in \cite{erdosschleinpeche} is
\begin{align*}
\frac{1}{n}\mathcal{K}_n\left(x,x+\frac{\tau}{n\rho_{sc}(x)}\right)&=\frac{e^{(x^2-y^2)/(2a^2)}}{n}K_n^{a^2/n}\left(x,x+\frac{\tau}{n\rho_{sc}(x)}\right)\\
&=\rho_{sc}(x)\frac{\sin \pi \tau}{\pi \tau}(1+o(1))+O(n^{-1/2}).
\end{align*}
 uniformly for $\tau$ in a compact set, $|x|<\sqrt{1}-\kappa$, and $a^2$ of order $n^{-1+\delta}$ for $\kappa>0$ and $\delta>0$. We take $\tau=0$ and fix the size $a^2$ of the Gaussian part to be $a^2=\frac{1}{4}$. On the other hand, we want to allow values of $x$ such that $|x|\le \sqrt{2}-n^{-1/3+\delta_2}$. This is the statement of Lemma \ref{lem:diagonalestimate}. The only significant modification we must make to the argument in \cite{erdosschleinpeche} lies in the localization of the saddle points of $f_n(s)$ and $f_n(w)$ in \eqref{eq: johanssonform}. In \cite{erdosschleinpeche}, section 3.2, it is shown that the equation
 \begin{equation}\label{eq: saddlepointeq}
 f_n'(s)=4(s-x)+\frac{1}{n}\sum_{j=1}^n\frac{1}{s-z_j}=0
 \end{equation}
 has two conjugate complex roots $s_n^\pm$, which are approximately equal to the roots $s^\pm$ of
 \begin{equation}\label{eq: exactsaddlepteq}
 f'(s) = 4(s-x)+2(s-\sqrt{s^2-1}).
 \end{equation}
The latter are given by
\begin{equation}
\label{eq: exactroots}
s^\pm=\frac{3}{4}x\pm i\frac{1}{4}\sqrt{2-x^2}.
\end{equation} 
The information we need about $s_n^\pm$ is contained in the following
\begin{lemma}\label{lem: contraction}
Let $\delta_2>0$. For any $z\in Z_n$ and $x\in \mathbb{R}$ such that
\[|x|\le \sqrt{2}-n^{-1/3+\delta_2},\]
the equation
\[f_n'(s)=0\]
has two conjugate roots $s_n^+$ and $s_n^-=\overline{s_n^+}$. For large enough $n$, these roots satisfy:
\begin{align}
|s_n^\pm-s^\pm|\le n^{-1/2},\label{eq: saddleapprox}\\
\Im s^+_n > n^{-1/6+\delta_2/3}. \label{eq: onequarter}
\end{align} 
\begin{proof}
We rewrite \eqref{eq: saddlepointeq} as:
\[s=F_n(s):=x-\frac{1}{4n}\sum_{j=1}^n\frac{1}{s-z_j}.\]
Following \cite{erdosschleinpeche}, we show that for large $n$ the mapping $F_n$ is a contraction on
\[\Xi = \{s: |s-s^+|\le n^{-1/2}\}.\]
This will immediately imply the rest of the lemma. It is clear from the explicit expression \eqref{eq: exactroots} that
$\Im s^+> cn^{-1/6+\delta_2/2}$ if $|x|<\sqrt{2}-n^{-1/3+\delta_2}$.

Note that $s^+$ is a solution of the fixed point equation
\[s=F(s):= x-\frac{1}{2}(s-\sqrt{s^2-1}).\]
The determination of the square root is chosen such that $\sqrt{s^2-1}\sim s $ at infinity. This can be written as $\sqrt{s^2-1}=\sqrt{s-1}\sqrt{s+1}$ with the last two instances of $\sqrt{\cdot}$ denoting the principal determination.
The condition $\Im s^+> n^{-1/6+\delta_2/3}$ and the local semicircle law together imply:
\[|F_n(s)-F(s)|= O(n^{-5/6})\]
for $s\in \Xi$.
Moreover, a simple calculation shows that
\[|F'(s^+)|\le 1-12n^{-1/3+\delta_2}\]
for $x$ in the range of interest. By Cauchy's estimate, for $s\in \Xi$, we have
\begin{align*}
F'_n(s)&= \frac{1}{4n}\sum_{j=1}^n\frac{1}{(s-z_j)^2}\\
&= F'(s) + O(n^{-2/3-\delta_2/3})\\
&= -\frac{1}{2}\left(1-\frac{s}{\sqrt{s^2-1}}\right)+ O(n^{-2/3-\delta_2/3}).
\end{align*}
Thus $|F_n'(s)|\le 1-O(n^{-2/3})$ for $s\in \Xi$.
The above implies that
\begin{align*}
\sup_{\Xi}|F_n(s)-s^+|&=\sup_{\Xi}|F(s)-s^+|+O(n^{-5/6})\\
&= \sup_{\Xi} |F(s)-F(s^+)|+O(n^{-5/6})\\
&= \sup_{\Xi}|F'(s)|n^{-1/2}+O(n^{-5/6})\\
&= (1-cn^{-1/3+\delta/2})n^{-1/2}+O(n^{-5/6})\\
&< (1-c'(n))n^{-1/2}
\end{align*}
for some $0<c'(n)<1$ when $n$ is large enough. This establishes the claim.
\end{proof}
\end{lemma}
With \eqref{eq: onequarter} in hand, we can follow the procedure in \cite{erdosschleinpeche}, Section 3.4, to localize the main contribution to the double integral to a neighborhood of size $\epsilon=n^{-1+\delta}$ of the four points $w=s^{\pm}_n$ and $s=s^\pm_n$ and obtain Lemma \ref{lem:diagonalestimate}. Some estimates have to be carried out differently, because the parameter $t$ appearing in \cite{erdosschleinpeche} is taken to be small, while it is equal to $a^2=\frac{1}{4}$ in our case, but this does not present any serious difficulty.

We turn to the proof of the off-diagonal estimate \eqref{eqn: offdiagbound}. To obtain a bound on the quantity $K^{1/4n}_n(x,y)K^{1/4n}_n(y,x)$, we find it convenient to work with the integral representation \eqref{eq: crossingcontours}, whose form is somewhat more symmetric in $x$ and $y$. Although the integral is more singular, the main contribution to the two integrals is localized away from the singularity for those pairs $x$, $y$ which concern us. Indeed, the next lemma shows that the two saddle points are separated by a distance at least a constant times $|x-y|$:
\begin{lemma}\label{lem: separation}
There is a constant $C>0$ such that for each $n$ and $z\in Z_n$, the solutions $w_n^+$ and $s_n^+$ of the saddle point equations $f'_{n,x}(w_n^+)=0$ and $f'_{n,y}(s_n^+)=0$ with positive imaginary part satisfy:
\[ |w_n^+-s_n^+|\ge  C\cdot |x-y|\]
for $|x-y|\ge n^{-1+\delta}$ and $|x|,|y|>\sqrt{2}-n^{-1/3+\delta_2}$.
\begin{proof}
The existence of two conjugate solutions $w_n^\pm$, resp. $s_n^\pm$ to each of the saddle point equations follows as in Lemma \ref{lem: contraction}. When $|x-y|\ge 10n^{-1/2}$, the lemma then follows immediately from \eqref{eq: saddleapprox}. If $x$ and $y$ are closer, we start from the saddle point equations 
\begin{align*}
f'_{n,x}(w)&= 4(w-x)+\frac{1}{n}\sum_{j=1}^n\frac{1}{w-z_j}=0,\\
f'_{n,y}(s)&=4(s-y)+\frac{1}{n}\sum_{j=1}^n\frac{1}{s-z_j}=0.
\end{align*}
Subtracting these two equations, we obtain
\[w-s=x-y+(w-s)\cdot \frac{1}{4n}\sum_{j=1}^n\frac{1}{(s-z_j)(w-z_j)}.\]
Rearranging, we find,
\begin{equation}\label{eq: approximateequation}
(w-s)\left(1-\frac{1}{4n}\sum_{j=1}^n\frac{1}{(s-y_j)(w-y_j)}\right)=x-y.\end{equation}

\[\frac{\mathrm{d}}{\mathrm{d}t}\frac{1}{(s-t)(w-t)}=-\frac{1}{(s-t)^2(w-t)}-\frac{1}{(s-t)(w-t)^2}.\]
Thus, by \eqref{eq: onequarter}, we have
\[\|1/((z-\cdot)(w-\cdot))\|_{C^1(\mathbb{R})}\lesssim n^{1/2-3\delta_2/2}.\]
The assumption $z\in Z_n$ allows us to write:
\begin{align*}\frac{1}{n}\sum_{j=1}^n\frac{1}{(s-z_j)(w-y_j)}&=\int \frac{1}{(s-t)(w-t)}\,\rho_{n,1}(\mathrm{d}t)\\
&= -\int \mathfrak{n}(x)\frac{\mathrm{d}}{\mathrm{d}t}\frac{1}{(z-t)(w-t)}\,\mathrm{d}t\\
&= \int \frac{1}{(s-t)(w-t)}\,\rho_{sc}(\mathrm{d}t) +\|1/((s-\cdot)(w-\cdot))\|_{C^1(\mathbb{R})}\cdot O(n^{-1+\delta})
\end{align*}
where $\mathfrak{n}(x) =\int_{\sqrt{2}}^x\rho_{1,n}(\mathrm{d}t).$
It remains to compute the quantity
\[ \int \frac{1}{(s-t)(w-t)}\,\rho_{sc}(\mathrm{d}t)\]
For this purpose, it is useful to develop the integrand into partial fractions:
\[\frac{1}{(s-t)(w-t)}=\frac{1}{w-s}\frac{1}{s-t}+\frac{1}{s-w}\frac{1}{w-t}.\]
Computed separately, the two integrals on the right result in the Stieltjes transform of the semicircle distribution, evaluated at $s$, resp. $w$. Thus we have the identity:
\begin{align*}
\int\frac{1}{(s-t)(w-t)}\,\rho_{sc}(\mathrm{d}t) &= \frac{2}{w-s}(s-\sqrt{s^2-1}-w+\sqrt{w^2-1})\\
&=2\left(-1+\frac{w+s}{\sqrt{w^2-1}+\sqrt{s^2-1}} \right).
\end{align*}
Using \eqref{eq: saddleapprox}, we have
\[\sqrt{(s_n^+)^2-1}=\sqrt{(s^+)^2-1}+\sup_{|\alpha|\le n^{-1/2}}\frac{|s^++\alpha|}{|\sqrt{(s^++\alpha)^2-1}|}O(n^{-1/2}).\]
By \eqref{eq: exactroots}:
\[|\sqrt{(s^+)^2-1}|=\frac{1}{4}[(10x^2-18)^2+36x^2(2-x^2)]^{1/4}.\]
It is readily checked that the function $x\mapsto (10x^2-18)^2+36x^2(2-x^2)$ is symmetric, positive on $[0,\sqrt{2}]$, and strictly decreasing on that interval. At $x=\sqrt{2}$, it takes the value $4$. Thus, we have:
\begin{equation}\label{eq: aandb} \frac{w_n^++s_n^+}{\sqrt{(w_n^+)^2-1}+\sqrt{(s^+_n)^2-1}}= \frac{w^++s^+}{\sqrt{(w^+)^2-1}+\sqrt{(s^+)^2-1}} +O(n^{-1/2}).
\end{equation}
Since $w^+ = s^+ +O(n^{-1/3-\delta_2/2})$, the left side of \eqref{eq: aandb} can be approximated as
\[\frac{s^+}{\sqrt{(s^+)^2-1}}+O(n^{-1/3-\delta_2/2}).\]
The real part of this quantity is bounded above by $3$, and is strictly smaller when $|x|<\sqrt{2}$ (see \eqref{eq: exactratio}  below). 
The equation \eqref{eq: approximateequation} becomes
\begin{equation*}
(w-s)\left(\frac{3}{2}-\frac{s^+}{2\sqrt{(s^+)^2-1}}+O(n^{-1/3-\delta_2/2})\right)=x-y.\end{equation*}
The result follows by taking absolute values on both sides of the last equation. Moreover, if $|x|,|y|\le \sqrt{2}-\epsilon_0$ for a fixed $\epsilon_0>0$, then we also have the reversed inequality:
\[|s^+-w^+|\le C_{\epsilon_0}|x-y|.\]
\end{proof}
\end{lemma}
The final result of this section is the following
\begin{proposition}
Let $\delta,\delta_2>0$ and $z\in Z_n$. There is a constant $C_{\delta,\delta_2}$ such that, for any $x$ and $y$ satisfying
\begin{align*}
|x-y|&\ge n^{-1+\delta},\\
|x|, |y|&\ge \sqrt{2}-n^{-1/3+\delta_2},
\end{align*}
we have the estimate:
\begin{equation}
\label{eq: productbound}
|K^{1/4n}_n(x,y;z)K^{1/4n}_n(y,x;z)|\le C_{\delta,\delta_2}|x-y|^{-2}.
\end{equation}
\begin{proof}
We will show that the product $|K^{1/4n}_n(x,y;z)||K^{1/4n}_n(y,x;z)|$ is bounded by the minimum of two quantities, one of which bounded by the left side of \eqref{eq: productbound} for 
\[|x-y|>100n^{-1/2},\] the other being bounded by the left side the equation when 
\[|x-y|\le 100n^{-1/2}.\] In either case, the starting point is the contour integral \eqref{eq: crossingcontours}. If, in this integral, we move the contour $\Gamma_L$ into $\Gamma_{L'}$ with $0<L'<M$, we obtain:
\begin{equation}\label{eq: sineb}
K_n^{1/4n}(x,y) = \frac{e^{2n(x^2-y^2)}n}{(\pi i)^2 }\int_\gamma\int_{\Gamma_{L'}} \frac{1}{w-s}e^{n(f_{n,x}(w)-f_{n,y}(s))}\,\mathrm{d}w\mathrm{d}s+e^{4nL'(x-y)}\frac{\sin(4nb(x-y))}{\pi(x-y)}.
\end{equation}

The assumption $|x-y|\ge 100n^{-1/2}$ implies, per  \eqref{eq: saddleapprox} , $|\Re(w_n^+-s_n^+)|\ge 50n^{-1/2}$. This will allow us to choose, for each $n$, $b$ such as to make the second term in \eqref{eq: sineb} vanish, and to extend the integral over $\gamma$ to an integral over an unbounded contour $\tilde{\gamma}(b)$, the union of nearly horizontal curves extending from $-\infty$ to $\infty$, traversed in opposite directions. Indeed, we can find a solution $b$ of 
\begin{equation}\label{eq: sinezero}
\sin(4nb(x-y))=0
\end{equation}
within distance $\frac{n^{-1/2}}{100}$ of any horizontal line (Figure \ref{fig:cont2}). 
\begin{figure}[h]
\centering
\begin{tikzpicture}
\draw[->] (-5,0) -- (5,0); 
\draw[->] (0,-5) -- (0,5); 
\draw[->] (5,2) -- (1,2); 
\draw (1,2) [rounded corners = 3pt] -- (-1,2) -- (-1.8,1) -- (-5,1);
\draw (-5,-1) [rounded corners = 3pt] -- (-1.8,-1) -- (-1,-2) -- (0.7,-2); 
\draw[->] (0.7,-2) -- (1,-2);
\draw (1,-2) -- (5,-2);
\draw[->] (3,-5)--(3,0.7); 
\draw (3,0.7)--(3,5);
\tiny 
\draw (3.2,1) node{$\Gamma$};
\fill [black] (0,2) circle (1pt);
\fill [black] (0,-2) circle (1pt);
\draw (0.2,2.2) node{$ib$};
\draw (0.2,-2.2) node{$-ib$};
\node (w) at (3.3,4) {$w_n^+$}; 
\fill [black] (3,4) circle (1pt);
\node (s) at (-2,1.3) {$s_n^+$};
\fill [black] (-2,1) circle (1pt);
\end{tikzpicture}
\caption{Contour for the case $|x-y| \geq 100n^{-1/2}$}
\label{fig:cont2}
\end{figure}
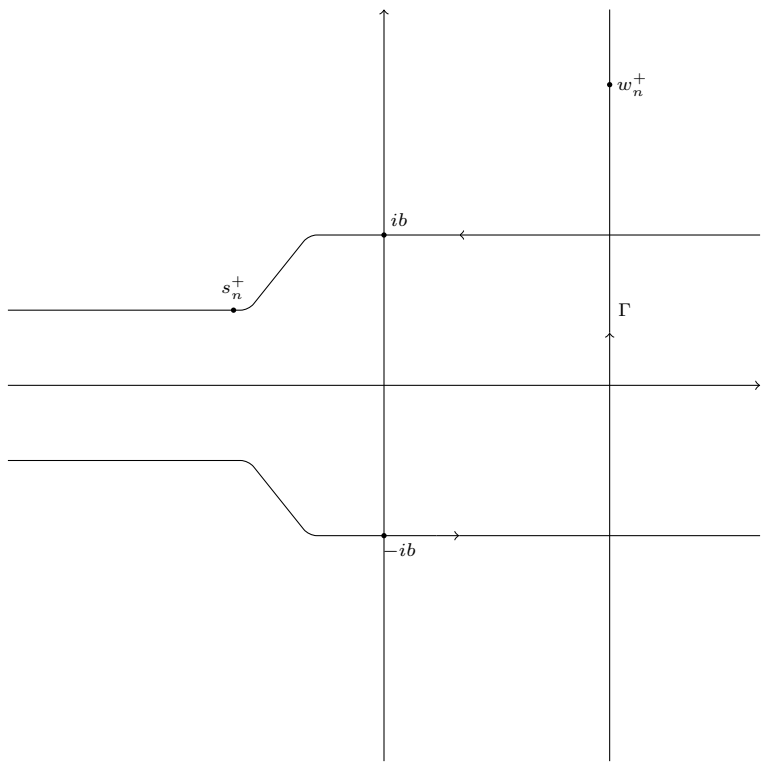

For the purposes of illustration, we assume $\Re s_n^+ < \Re w_n^+$. We first choose $L=\Re w_n^+$, and let the portion of $\tilde{\gamma}$ in the upper half-plane be 
\begin{align*}
\tilde{\gamma}_1: t\mapsto &\ i\Im s_n^+ + t, \quad -\infty < t\le \Re s_n^+-\epsilon\\
&i b(n) + t,  \quad \  \Re w_n^+ \le t <\infty,
\end{align*}
here $b(n)$ is the solution of \eqref{eq: sinezero} closest to $\Im s_n^+$. We choose the part $\tilde{\gamma}^*_1$ of $\tilde{\gamma}_1$ outside an $\epsilon$-neighborhood of $s_n^+$ and between $\Re s_n^+$ and $\Re w_n^+$ to be a smooth curve $\tilde{\gamma}^*_1(t)$, $0 \le t\le 1$ with increasing real part, and monotone imaginary part, and with imaginary part always lying between $\Im s_n^+$ and  $b(n)$. The component $\tilde{\gamma}_2$ of $\tilde{\gamma}$ in the lower half plane is obtained by reflection about the real axis. We let
\[m = |\Im s_n^+ -b(n)|.\]
Our assumptions imply that 
\[k= |\Re w_n^+- \Re s_n^+|\ge 50m.\]
The contribution from neighbourhoods of size $\epsilon=\frac{1}{3}|\Re s_n^+ - \Re w_n^+|$ of the saddle points $s_n^\pm$ and $w_n^\pm$ to the integral
\[K_n^{1/4n}(x,y) = \frac{e^{2n(x-y)^2}n}{(\pi i)^2 }\int_{\tilde{\gamma}}\int_{\Gamma} \frac{1}{w-s}e^{n(f_{n,x}(w)-f_{n,y}(s))}\,\mathrm{d}w\mathrm{d}s\]
are now obtained by a standard Laplace approximation, as in \cite{erdosschleinpeche}. We find that that $K_n^{1/4n}(x,y)$ is given by the sum of four terms of the form
\begin{align*}
\frac{e^{2n(x^2-y^2)}}{(i\pi)^2}\frac{2\pi}{\sqrt{f''_{n,x}(w_n^a)}\sqrt{f''_{n,y}(s_n^b)}}\frac{1}{w_n^a-s_n^b}e^{n(f_{n,x}(w_n^a)-f_{n,y}(s_n^b))}\left(1+O(n^{-1/2})\right),
\end{align*}
for all four choices $a,b\in \{\pm\}$. Multiplying this sum by the corresponding approximation for $K_n^{1/4n}(y,x)$, the real parts of all exponentials involved cancel, since
\begin{align}
f_{n,x}(w_n^+)&=\overline{f_{n,x}(w_n^-)} \label{eq: symmetry}\\
f_{n,y}(s_n^+)&=\overline{f_{n,y}(s_n^-)}.\nonumber
\end{align}
A simple calculation, using also the approximations in Lemma \ref{lem: contraction}, shows that $f_{n,x}''(w_n^\pm)$, $f_{n,y}(s_n^\pm)$ are bounded below. Finally, by Lemma \ref{lem: separation}, we have
\[\frac{1}{|w_n^a - s_n^b|}\lesssim \frac{1}{|x-y|}\] 
for any choice of signs $a,b\in \{\pm\}$. It follows that the contribution to
\[|K_n^{1/4n}(x,y)K_n^{1/4n}(y,x)|\]
from $\epsilon$-neighborhoods of the saddle points is bounded by $|x-y|^{-2}(1+O(n^{-1/2}))$. It remains to estimate the  contour integral $K_n^{1/4n}(x,y)$ for $s$, $w$ at distance greater than $\epsilon$ from $w_n^\pm$, $s_n^\pm$. We will show that the integrals over these portions of the contours give a contribution bounded up to a constant factor by
\[\frac{1}{|x-y|}e^{n(\Re f_{n,x}(w_n^+)-\Re f_{n,y}(s_n^+))}.\]
Using the symmetry \eqref{eq: symmetry}, the real parts of the exponential factor will cancel with the corresponding factors from $K^{1/4n}_n(y,x)$.
Define the set
\[\Omega = \{s=\alpha+i\beta: |\alpha-\Re s_n^+|\ge \epsilon, |\beta-\Im s_n^+|\le 50^{-1}\epsilon\}.\]
The part of the contour $\tilde{\gamma}_1$ outside of a ball of radius $\epsilon$ around $s_n^+$ lies entirely in $\Omega$. For $s\in \Omega$, we have:
\[\Re[f_{n,y}(\alpha+i\beta)-f_{n,y}(s_n^+)]\ge c(\alpha-\Re s_n^+)^2\]
for some constant $c>0$. A computation using $f'(s^+)=0$ shows that for $\alpha+i\beta \in \Omega$,
\[\Re f'(\alpha+i\beta) \ge c(\alpha-\Re s^+).\]
Thus 
\begin{align*}
\partial_x \Re f_{n,y}(\alpha+i\beta)&\ge \Re f'(\alpha+i\beta) - O(n^{-5/6})\\
&\ge \Re f'(s^+) -O(n^{-2/3})\\
&\ge c(\alpha -\Re s^+)\\
&\ge c'(\alpha - \Re s_n^+), \quad \alpha+i\beta \in \Omega.
\end{align*}
Integrating then gives the desired bound.  We also get a similar lower bound along the contour $\Gamma$ (again, see \cite{erdosschleinpeche}).
When we integrate outside the saddle point, one of the terms we have to deal with is of the form
\[\int_\Gamma \mathrm{d}w \int_{\widehat{\gamma}} \mathrm{d}s e^{n(f_{n,x}(w)-f_{n,y}(s))} \frac{1}{w-s}, \]
where $\widehat{\gamma}$ denotes the part of the contour $\tilde{\gamma}_1$ outside a neighborhood of $s^+_n$ of size $\epsilon$.  We will derive the bound for the above integral, the case where the $w$ lies outside an $\epsilon$-neighborhood of the saddle point $w_n^+$ is essentially the same.
Using the bound derived above, we have
\begin{align*}
&\left|\int_\Gamma \mathrm{d}w \int_{\hat{\gamma}} \mathrm{d}s e^{n(f_{n,x}(w)-f_{n,y}(s))} \frac{1}{w-s}\right| \\
& \leq \int_\Gamma \mathrm{d}w  e^{n(f_{n,x}(w)-f_{n,y}(s_n^+))} \int_{\hat{\gamma}} \mathrm{d}s e^{-cn(\Re s - \Re s_n^+)^2} \left|\frac{1}{w-s}\right|.
\end{align*}
Split $\widehat{\gamma}$ into $\widehat{\gamma}_o$ and $\widehat{\gamma}_i$, where $\widehat{\gamma}_o$ is the part outside of the (integrable) singularity, and $\widehat{\gamma}_i$ is the part within $\frac{1}{100} |\Re (w_n^+ - s_n^+)|$ of the singularity.  The above integral can be expressed as:
\begin{align}\label{eq: error1}
& \int_\Gamma \mathrm{d}w  e^{n(f_{n,x}(w)-f_{n,y}(s_n^+))} \int_{\widehat{\gamma}_o} \mathrm{d}s e^{-cn(\Re s - \Re s_n^+)^2} \left|\frac{1}{w-s}\right| \\ \label{eq: error2}
&+ \int_\Gamma \mathrm{d}w  e^{n(f_{n,x}(w)-f_{n,y}(s_n^+))} \int_{\widehat{\gamma}_i} \mathrm{d}s e^{-cn(\Re s - \Re s_n^+)^2} \left|\frac{1}{w-s}\right|
\end{align}
For the first term \eqref{eq: error1}, we use that $\frac{1}{|w-s|} \leq \frac{100}{|\Re (w_n^+ - s_n^+)|}$ on the contour, and 
\[\int_{\widehat{\gamma}_o}e^{-cn(\Re s - \Re s_n^+)^2}\mathrm{d}s \lesssim n^{-1/2}.\] 
We obtain an additional factor of $n^{-1/2}$ for $\Gamma$ from either the $\epsilon$-neighborhood saddle point $w_n^+$ or the gaussian decay of the $w$ integral away from $w_n^+$. We conclude that the term \eqref{eq: error1} is bounded by
\[100 n^{-1}|\Re (w_n^+ - s_n^+)|^{-1} e^{n\Re (f_{n,x}(w_n^+) - f_{n,y}(s_n^+))} \lesssim 100 n^{-1}|x-y|^{-1}e^{n\Re (f_{n,x}(w_n^+) - f_{n,y}(s_n^+))}.\]
For the second term \eqref{eq: error2}, notice that since $|\Re s - \Re w| \leq \frac{1}{100}|\Re (w_n^+-s_n^+)|$, and $w\in\Gamma$ has constant real part,
we have
\begin{align*}
|\Re s - \Re s_n^+| &= |\Re w_n^+ - \Re s_n^+ + \Re s - \Re w| \\
&\geq |\Re w_n^+ - \Re s_n^+| - \frac{1}{100}|\Re (w_n^+-s_n^+)| \\
& \geq \frac{1}{2}|\Re (w_n^+-s_n^+)|.
\end{align*}
Therefore, we have $|\Re s - \Re s_n^+|^2 \geq \frac{1}{4}|\Re (w_n^+-s_n^+)|^2 $.  The singularity $|\frac{1}{w-s}|$ behaves like $1/r$ in two dimensions, therefore the second term is bounded by
\[ e^{n\Re (f_{n,x}(w_n^+) - f_{n,y}(s_n^+))} \frac{1}{100} |\Re w_n^+- \Re s_n^+| e^{-cn|\Re (w_n^+-s_n^+)|^2}.\]
From the fact that
\[ x^2e^{-nx^2} \lesssim n^{-1}\]
for $x \in \mathbb{R}$, we conclude the estimate:
\begin{align*}
e^{\Re n(f_{n,x}(w_n^+) - f_{n,y}(s_n^+))} \frac{1}{100} |w_n^+-s_n^+| e^{cn|\Re (w_n^+-s_n^+)|^2} &\lesssim e^{n\Re (f_{n,x}(w_n^+) - f_{n,y}(s_n^+))} n^{-1}|\Re (w_n^+-s_n^+)|^{-1} \\
&\lesssim e^{n\Re (f_{n,x}(w_n^+) - f_{n,y}(s_n^+))} n^{-1}|x-y|^{-1}.
\end{align*}
In the case $|x-y|\le 100n^{-1/2}$, we do not attempt to set the sine term in \eqref{eq: sineb} to zero. We choose $\tilde{\gamma}$ to consist of two horizontal lines through $s^+_n$ and $s^-_n$, respectively, traversed with opposite orientations (Figure \ref{fig:cont3}).
We can then reproduce the saddle point analysis for $K_n^{1/4n}(x,y)$ and $K_n^{1/4n}(y,x)$. When we multiply the resulting approximations, we find 16 terms similar to the ones encountered above, each of which is bounded, up to constant factors by $\frac{1}{|x-y|^2}$. There are also 8 terms of the form
\[\frac{c^{a,b}_n}{w_n^a-s_n^b }\frac{\sin(2n\Im s_n^+(x-y))}{\pi(x-y)}e^{4n(y-x)\Re w_n^+ }e^{n(f_{n,y}(s_n^b)-f_{n,x}(w_n^a))},\]
where $a,b\in \{\pm\}$ and $c_n^{a,b}$ are bounded independently of $n$. All these terms are bounded by $\frac{1}{|x-y|^2}$ up to a constant factor. To see this, it clearly suffices to show that
\begin{equation}\label{eq: realpartbound}
n \cdot (4(y-s)\Re w_n^+ + \Re f_{n,y}(s_n^b)-\Re f_{n,x}(w_n^a))
\end{equation}
is bounded by a constant independent of $n$, for any choice of signs $a,b \in \{\pm\}$.
Notice the identities
\[\overline{f_{n,x}(w)} =f_{n,x}(\overline{w}),\]
and
\[\overline{f_{n,y}(s)} =f_{n,y}(\overline{s}),\]
so that we have
\begin{align*}
\Re f_{n,x}(w_n^+) &= \Re f_{n,x}(w_n^-)\\
\Re f_{n,y}(s_n^+) &= \Re f_{n,y}(s_n^-).
\end{align*}
Thus it suffices to deal with the case $a=b=+$ in \eqref{eq: realpartbound}. Write
\begin{align*}
f_{n,y}(s)-f_{n,x}(w) &= 2(s^2-2ys -w^2+ 2xw) +\frac{1}{n}\sum_{j=1}^n\log\left(1+\frac{s-w}{w-z_j}\right)\\
&= 2(s^2-2ys -w^2+ 2xw)+ (s-w)\cdot \frac{1}{n}\sum^n_{j=1}\frac{1}{w-z_j} \\
&\quad - \frac{(s-w)^2}{2n}\sum_{j=1}^n\frac{1}{(w-z_j)^2} +\frac{1}{n}\sum_{j=1}^n \sum_{k\ge 3} (-1)^{k-1}\left(\frac{s-w}{w-z_j}\right)^k
\end{align*}
When $w=w_n^+$, we can use the saddle point equation \eqref{eq: saddlepointeq} to replace $\frac{1}{n}\sum_{j=1}^n\frac{1}{w-y_j}$ in the above equation by
\[4x-4w_n^+.\]
Thus:
\begin{align}\label{eq: fndiff}
f_{n,y}(s_n^+)-f_{n,x}(w_n^+) &= 2(s_n^+-w_n^+)^2+4(x-y)s_n^+\\
&\quad - \frac{(s_n^+-w_n^+)^2}{2n}\sum_{j=1}^n\frac{1}{(w_n^+-z_j)^2} +\frac{1}{n}\sum_{j=1}^n \sum_{k\ge 3} (-1)^{k-1}\left(\frac{s_n^+-w_n^+}{w_n^+-z_j}\right)^k. \nonumber
\end{align}
Taking the real part of the first line on the right, we find
\begin{equation}
\label{eq: goodterms}
2\left((\Re(s_n^+)-\Re(w_n^+))^2-(\Im(s_n^+)-\Im(w_n^+))^2\right) +4(x-y)\Re s_n^+.
\end{equation}
Note that the term involving imaginary parts is always negative. We will find below that it cancels another, positive and real, term in the sum.
Using \eqref{eq: saddleapprox}, we have:
\begin{align*}
\Re(s_n^+)-\Re(w_n^+) &= \frac{3}{2}(x-y) + O(n^{-1/3}) = O(n^{-1/2}),\\
\Im(s_n^+)-\Im(w_n^+) &= \frac{1}{2}(\sqrt{2-y^2}-\sqrt{2-x^2})+O(n^{-1/2}) = O(n^{-1/3-\delta_2/2}).
\end{align*}
By \eqref{eq: goodterms}, we see that we can estimate \eqref{eq: realpartbound} provided that we can control the real parts of all terms appearing on the second line of \eqref{eq: fndiff}. 
Since we have
\begin{align*}
|s_n^+-w_n^+|&\lesssim n^{-1/3-\delta_2/2}\\
|w_n^+-z_j|&\ge \Im w_n^+ > n^{-1/6+\delta_2/2},
\end{align*}
we can easily estimate the sum
\begin{align*}\left|\frac{1}{n}\sum_{j=1}^n \sum_{k\ge 3} (-1)^{k-1}\left(\frac{s_n^+-w_n^+}{w_n^+-z_j}\right)^k\right| &\lesssim \sum_{k\ge 3} (Cn^{-1/3-\delta_2})^k\\
&\lesssim n^{-1-3\delta_2}.
\end{align*}
It remains only to deal with the term
\[- \frac{(s_n^+-w_n^+)^2}{2n}\sum_{j=1}^n\frac{1}{(w_n^+-z_j)^2}.\]
The potentially dangerous terms comprise the real part of the product above:
\begin{gather}
-\frac{1}{2}\left((\Re(s_n^+)-\Re(w_n^+))^2-(\Im(s_n^+)-\Im(w_n^+))^2\right)\Re \frac{1}{n}\sum_{j=1}^n\frac{1}{(w_n^+-z_j)^2}, \label{eq: danger1}\\
(\Re s_n^+-\Re w_n^+)(\Im s_n^+-\Im w_n^+)\Im \frac{1}{n}\sum_{j=1}^n\frac{1}{(w_n^+-z_j)^2} \label{eq: danger2}.
\end{gather}
As in the proof of Lemma \ref{lem: contraction}, we have the approximation
\begin{align}
\frac{1}{2n}\sum_{j=1}^n\frac{1}{(w_n^+-z_j)^2} &= -1+\frac{w_n^+}{\sqrt{(w_n^+)^2-1}} +O(n^{-1/3-\delta_2/3}) \nonumber \\
&= -1+\frac{s^+}{\sqrt{(s^+)^2-1}}+O(n^{-1/2}),\label{eq: approxagain}
\end{align}
where $s^+$ is as in $\eqref{eq: exactroots}$. We examine the imaginary and real parts of the quantity \eqref{eq: approxagain}. By \eqref{eq: exactsaddlepteq}, we have:
\begin{equation}
\sqrt{(s^+)^2-1}=3s^+-2x.
\end{equation}
Thus, we have:
\begin{align}\label{eq: exactratio}
\frac{s^+}{\sqrt{(s^+)^2-1}} = \frac{s^+}{3s^+-2x}=\frac{3x^2+3\sqrt{2-x^2}}{x^2+9(2-x^2)}-i\frac{9x\sqrt{2-x^2}}{x^2+9(2-x^2)}
\end{align}
The real part is a symmetric, positive function which increases from $0$ to $3$ on the interval $[0,\sqrt{2}]$. The imaginary part is an anti-symmetric, negative for $0<x<\sqrt{2}$ and zero at the endpoints. 

For \eqref{eq: danger1}, we find:
\begin{align*}
(\Im(s_n^+)-\Im(w_n^+))^2\cdot \left(\frac{3x^2+3\sqrt{2-x^2}}{x^2+9(2-x^2)} - 1\right)  +O(n^{-1})
\end{align*}
The factor in parentheses is at most $2$, and so the sum of the first term in the previous equation with the term $-2(\Im(s_n^+)-\Im(w_n^+)^2$ from \eqref{eq: goodterms} is at most zero.
For the term \eqref{eq: danger2}, we find:
\begin{align*}
 (\Re s_n^+-\Re w_n^+)(\Im s_n^+-\Im w_n^+)\Im \frac{1}{n}\sum_{j=1}^n\frac{1}{(w_n^+-z_j)^2} &=  -\frac{3}{4}(\Re s_n^+-\Re w_n^+)(\sqrt{2-x^2}-\sqrt{2-y^2})\\
 &\times\left(\frac{9x\sqrt{2-x^2}}{x^2+9(2-x^2)}\right)+O(n^{-1})\\
 &= \frac{3}{4}(\Re s_n^+-\Re w_n^+)(y-x)(y+x)\\
 &\times \frac{\sqrt{2-x^2}}{\sqrt{2-y^2}+\sqrt{2-x^2}} \frac{1}{x^2+9(2-x^2)} +O(n^{-1}).
\end{align*}
For  $2-y^2 = o(|x-y|)$, the ratio of radicals is bounded, and so the last quantity is $O(n^{-1})$. All the terms in \eqref{eq: realpartbound} are now accounted for, and the argument is closed.
\end{proof}
\end{proposition}

\section{CLT for $C^{1-\epsilon}$ and $H^{1-\epsilon}$: Wigner matrices}
In this section, we prove the CLT for general Wigner matrices under matching moments conditions at the $C^{1-\epsilon}$ level of regularity.
The proof follows similarly to that of the previous section.  Writing 
\begin{align*}
\varphi &= \left(\sum_{k\le (1+\delta)\log n} +\sum_{k>(1+\delta)\log n}\right)\varphi_k \\
&= \varphi^{1,n}+\varphi^{2,n},
\end{align*}
where $\delta$ is as in Proposition \ref{prop:comp}.
The variance of the first term is bounded as in Section \ref{sec: wigner}; we now compare $M$ with a GUE matrix with entries of complex variance $1$ rather than $1/\sqrt{2}$ times a GUE matrix.  For the second term, we use the trivial bound:
\begin{align*}
\mathbf{Var}(\mathcal{N}_n[\varphi^{2,n}])&\lesssim n^2 \|\varphi^{2,n}\|_{L^{\infty}} \\
&\lesssim n^2 \sum_{k >(1+\delta)\log n} 2^{-2k(1-\epsilon)} \|\varphi^{2,n}\|^2_{C^{1-\epsilon}}
\end{align*}
Picking $\epsilon$ sufficiently small, the above quantity is $o(1)$.  As in previous cases, this bound for the variance implies the CLT.

The proof for $H^{1-\epsilon}$ test functions supported away from the edge follows the same way as in the proof of the CLT for $H^{1/2+\epsilon}$ test functions in the Johansson matrices case.

 \end{document}